\documentclass[openacc]{rsproca_new_HAL}
\usepackage{lipsum,physics,mathtools,bbold,nccmath,scalerel,stmaryrd,array}
\usepackage{mathrsfs}
\usepackage{enumitem}
\usepackage[french, main=english]{babel} 
\usepackage[babel]{csquotes} 
\usepackage{tikz} 
\usetikzlibrary{decorations.markings, matrix}
\usepackage{multicol}
\usepackage[figurename=Fig.]{caption} 
\usepackage{subcaption}
\usepackage{float}

\usepackage{anyfontsize}

\SetSymbolFont{stmry}{bold}{U}{stmry}{m}{n}

\newtheorem{theorem}{\bf Theorem}[section]
\newtheorem{proposition}{\bf Proposition}[section]

\newtheorem{definition}{\bf Definition} 

\newcommand{\cM}{\mathcal{M}}
\newcommand{\R}{\mathbb{R}}
\newcommand{\N}{\mathbb{N}}
\newcommand{\cV}{\mathcal{V}}
\newcommand{\bP}{\mathbb{P}}
\newcommand{\bE}{\mathbb{E}}

\newcommand{\p}{\partial}
\newcommand{\mpdv}[1]{\tfrac{\p}{\p {#1}}}

\renewcommand{\leq}{\leqslant}
\renewcommand{\geq}{\geqslant}

\newcommand{\msum}[1]{\displaystyle\scaleobj{.98}{\,
    \sum_{\scaleobj{.91}{#1}}} \,}

\newcommand{\mprod}[1]{\displaystyle\scaleobj{.98}{\,
    \prod_{\scaleobj{.9}{#1}}} \,} %
\newcommand{\mmprod}[2]{\displaystyle\scaleobj{.98}{\,
    \prod_{\scaleobj{.9}{#1}}^{\scaleobj{.9}{#2}}} \,}

\DeclareMathOperator{\Var}{Var}
\DeclareMathOperator{\Cov}{Cov}
\DeclareMathOperator{\Cor}{Cor}

\newcommand{\texp}[1]{\hspace{-.5pt} \scaleobj{0.87}{{#1}}}

\def\lp{\left(}
\def\rp{\right)}
\newcommand{\f}{\frac}
\newcommand{\1}{\mathbb{1}}

\begin{document}

\makeatletter
\def\namedlabel#1#2{\begingroup
   \def\@currentlabel{#2}%
   \label{#1}\endgroup
}
\makeatother

\title{Comparison Between Effective and Individual Growth Rates in a Heterogeneous Population}

\author{
Marie Doumic$^{1}$, Anaïs Rat$^{2,3}$, Magali Tournus$^{2}$}

\address{$^{1}$
 CMAP, Inria, IP Paris, Ecole polytechnique, CNRS, 91128 Palaiseau cedex.\\
$^{2}$Aix Marseille Univ,
    CNRS, I2M, Centrale Marseille, Marseille, France
    \\
    $^{3}$Univ Brest, CNRS UMR 6205, Laboratoire de Mathématiques de Bretagne Atlantique, F-29200 Brest, France}

\subject{biomathematics}


\corres{Anaïs Rat\\
\email{anais.rat@cnrs.fr}}

\begin{abstract}
  Is there an advantage to heterogeneity in a population where individuals grow and divide by fission? 
  This is a broad question, to which there is no easy universal answer. 
  This article aims to provide a quantitative answer in the specific context of growth rate heterogeneity by comparing the fitness of homogeneous versus heterogeneous populations.
  We focus on size-structured populations, where the growth rate of each individual is set at birth by heredity and/or random mutations. 
  The \emph{fitness} (or \emph{Malthus parameter}, or \emph{effective} fitness) of such heterogeneous population is defined by its long-term behaviour, and we introduce the effective growth rate as the individual growth rate in the homogeneous population with the same fitness.
  We derive analytical formulae linking effective and individual growth rates in two paradigmatic cases: first, constant growth and division rates, second, linear growth rates and uniform fragmentation. Surprisingly, these two cases yield similar expressions.
  Then, by comparing the fitness and the effective growth rates of populations with different degrees of heterogeneity or different laws of heredity/mutation
  to those of average homogeneous populations, we quantitatively investigate the combined influence of heredity and heterogeneity, and revisit previous results stating that heterogeneity is beneficial in the case of strong heredity. 

\end{abstract}


\maketitle

\begin{multicols}{2}

\section{Introduction}
Heterogeneity in living populations has been evidenced and questioned in many different studies, for instance --- among many others --- variability in metabolism, genes and growth rate in bacteria~\cite{kiviet2014stochasticity,jouvet2018demographic,elowitz2002stochastic}, variability in protein production, morphology and growth rate in yeasts~\cite{newman2006single,levy2012bet}, heterogeneity in telomere lengths in eukaryotic cells~\cite{lansdorp1996heterogeneity,kachouri2009large}, etc. 

An important question for ecology is how to relate these variable traits to the concept of \textit{fitness}, either at a population or at an individual level: which trait is beneficial or detrimental to growth? Does variability increase or decrease this {fitness}? 

At the population level, the fitness may be defined as the Malthusian fitness, that is the exponential growth rate of the number of individuals, called the \textit{Malthus parameter} of the population. At an individual level however, it is uneasy to find which traits may be directly related to fitness~\cite{martin2016nonstationary}, except in an important and paradigmatic case: when individuals grow exponentially in size, the individual growth rate is equal to the Malthus parameter of a clonal population sharing the same trait~\cite{robert2018mutation}. Such an individual exponential growth, or elongation, is  the case for some bacteria --- e.g.\ \textit{E. coli} in normal growth condition.

This is one of the reasons why  the study of  heterogeneity in \textit{individual} growth rate is of central importance~\cite{lin2017effects,LinAmirFrom2019,genthon2020fluctuation,genthon_noisy_2024}. Other motivations are the study of heredity through mother-daughter correlation~\cite{delyon2018investigation,shi2020allocation}, and it has also been used as an individual fitness to study mutation effects~\cite{robert2018mutation}.

These considerations led us to the following question:
What are the effective growth rate and fitness of a population where individuals display different growth rates? In other words, does there exist an \emph{equivalent} homogeneous population, i.e.\ with the same fitness but where all individuals grow at the same speed? If so, how does this \enquote{effective} rate relates to the growth rates of the heterogeneous population?
This wide-ranging question has evolutionary interpretations, to understand for instance whether or not heterogeneity ensures not only more robust survival~\cite{levy2016cellular} or adaptation to a varying environment~\cite{magdanova2013heterogeneity}, but also inherently faster growth, or growth at a lower energetic cost, or yet if heritable heterogeneity is more or less advantageous than non-heritable traits~\cite{carja2017evolutionary,mattingly2022collective}. 
    
No universal answer can be provided, since it depends on the various traits of the population and how the individual growth is related to the population one~\cite{jouvet2018demographic,lidstrom2010role}.
In this article, we propose a methodological approach that we believe could be adapted to  other cases, and we apply it to size-structured populations where reproduction occurs by division into smaller individuals --- as is typically the case for microbial populations. This context allows us to derive exact formulae for the effective growth rate of the population as an explicit weighted average 
of the individual ones, in two  cases: first, constant growth and division rates, with general and possibly asymmetric size distribution between daughters; and second, linear growth rate and uniform size distribution between newborns, with general division rates. Surprisingly, we obtain almost identical analytical formulae for the weighted averages in these two cases. We also derive new explicit analytical formulae for the steady size distribution~\cite{genthon2022analytical}. 
We complete the study with a numerical investigation of the emblematic case of exponential individual growth and division into two cells of equal size, for which no analytical formula can be derived.

\begin{figure}[H]
  \centering
  \includegraphics[width=0.85\linewidth]{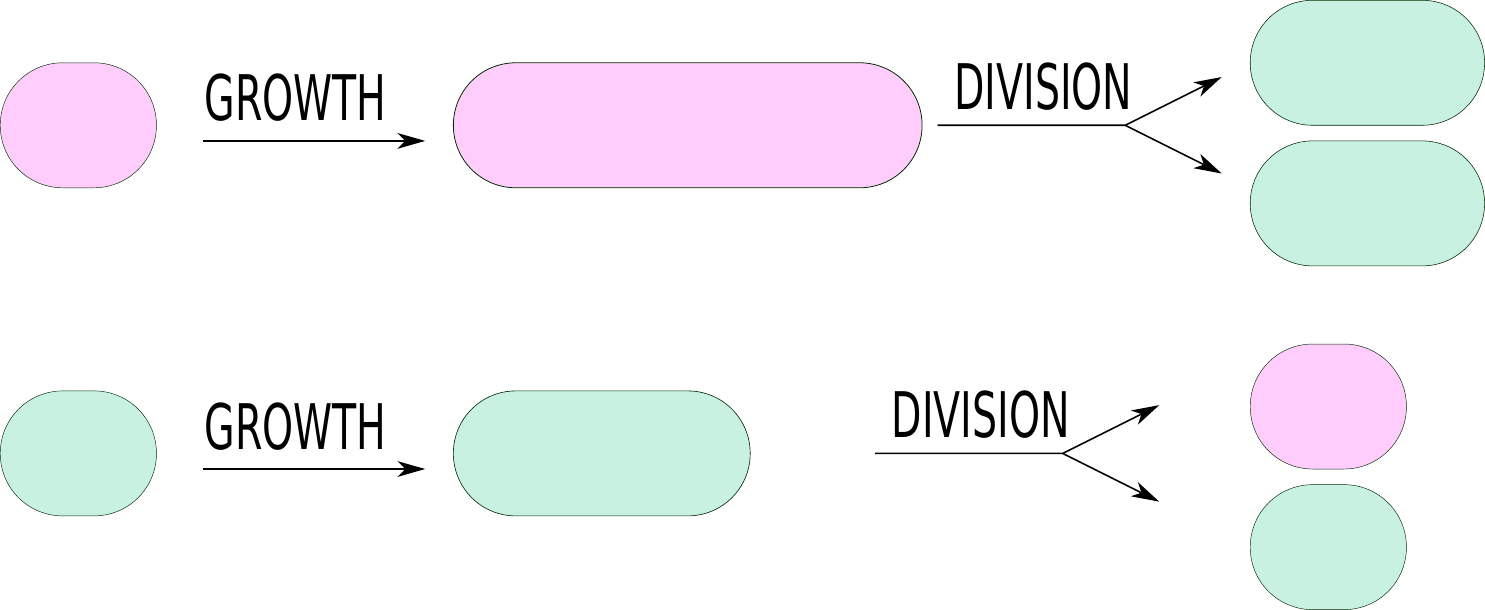}
  \caption{Schematic representation of a heterogeneous mixing population: pink cells grow faster than green ones, and dividing cells can give rise to both types according to a certain heredity law. For simplicity, equal mitosis (when a cell divides into two cells of equal size) is represented.}\label{fig:test}
\end{figure}

\section{Methods}\label{sec:methods}

\subsection{Definition of the effective growth rate}\label{subsec:method}

Let us first explain our approach in a  general setting. We consider a population where each  individual  grows  and  reproduces by division, independently from each other, and possibly in variable ways. Individuals are described by two traits:  a \enquote{master} trait $x$ --- think of the size, the physiological age, or any quantity that evolves through the individual's life and may influence their
probability to divide --- and an \enquote{individual} trait $v$, conserved through life, that characterises the growth of $x$ --- it could represent protein production rate, growth or ageing rate, phenotypic trait, gene expression, etc. This last trait is assumed to correspond uniquely to a growth rate, allowing us to use \enquote{trait} and \enquote{growth rate} interchangeably in what follows. It thus accounts for the heterogeneity in growth rates. For simplicity, let us consider this trait as a discrete variable, taking a finite number of values $M$, values that we denote $v_1 < \ldots < v_M$. 

An individual with trait $v_i$ gives birth to an individual with trait $v_j$ with a certain probability $\kappa_{ij}$. We call the \emph{heredity kernel} the matrix $\kappa=(\kappa_{ij})_{1\leq i,j\leq M}$, which is a stochastic irreducible matrix in $\cM_M(\R)$, i.e.\ it satisfies
\begin{gather}\label{as2:kappa}
  \forall\, 1 \leq i, j \leq M, \quad \kappa_{ij} \geq 0, 
  \quad \sum_{j=1}^{M} \kappa_{i j } = 1,\\
  \forall (i,j), \; \exists m \in \N^* \; : \; \kappa_{ij}^{(m)}
  \coloneqq \lp \kappa^m \rp_{ij} > 0.\label{as1:kappa}
\end{gather}
Assumption~\eqref{as2:kappa} ensures that for any $i$, $(\kappa_{i j})_{1\leq j\leq M}$ is a probability law, while  the irreducibility assumption~\eqref{as1:kappa} ensures that any ancestor of trait $i$ has  a non-zero probability to give birth to a descendant of  any trait $j$  in a finite number of divisions. 
If it fails, asymptotic behaviours may also be characterised but several Malthusian behaviours can emerge for unconnected subpopulations, and some subpopulations may also extinct. 

Under balance assumptions on the laws of growth and birth, it can be proven that the population reaches what we can call \emph{homeostasis}, defined here as a stable distribution of traits (see Fig.~\ref{fig:N} for an illustration). With this formalism, the question of the influence of heterogeneity and heredity can thus be formulated as follows: How can we compare this heterogeneous (in trait $v_i$) population with an homogeneous one? We notice that for each trait $v_i$ we can define a homogeneous population by assuming  that it shares all the characteristics of the heterogeneous population, except that the trait $v_i$ is common to all individuals and unchanged at birth --- in other words, the population is given by $v_i$ for its trait, $M=1$ and $\kappa=\kappa_{11}=1$. This results in the following definition.

\begin{definition}[Effective growth rate (or trait) of a heterogeneous population]\label{def2:ef_fitness}
  Consider a population composed of~$M$ sub-populations with traits~$\boldsymbol{v}=(v_1,v_2,\ldots, v_M)$ transmitted at birth according to a kernel $\kappa$ satisfying~\eqref{as2:kappa}-\eqref{as1:kappa}. We assume that growth and birth laws ensure the existence and uniqueness of a Malthusian growth at exponential rate $\lambda_{(\boldsymbol{v},\kappa)}>0$ and the convergence of the traits distribution towards a steady homeostatic profile, and similarly, in the case of a unique trait $v$, the existence and uniqueness of a Malthusian growth at a rate $\lambda_v$.
  An \emph{effective {growth rate}} or \emph{effective trait} of the heterogeneous population of traits~$\boldsymbol{v}$ is a trait~$v$ such that~$\lambda_v=\lambda_{(\boldsymbol{v},\kappa)}$ is the Malthus parameter of a homogeneous population of trait~$v$, i.e.\ a population with the same characteristics as the heterogeneous one, but with homogeneous trait $v$. 
\end{definition}

We emphasize that the \emph{effective} trait is, in many cases of interest, an increasing function of the Malthus parameter --- in our study, it is even proportional to it. From this perspective, each individual trait $v$ can be associated with the fitness of a homogeneous population of trait $v$, such that the trait distribution can be related to the \enquote{fitness landscape}~\cite{fragata2019evolution,nozoe2017inferring}, and the \emph{effective trait} is directly related to the \emph{effective fitness} (or Malthus parameter) of the population.

If the function $w \mapsto \lambda_w$ is strictly increasing and covers a sufficiently wide range of values, Definition~\ref{def2:ef_fitness} ensures the existence and uniqueness of the effective trait~$v$. To  better understand whether heterogeneity is beneficial or detrimental to the population growth, we propose to compare~$v$ with the distribution of traits $v_i$. As discussed in many articles~\cite{olivier2017does,10.1093/g3journal/jkad230}, the definition of neutrality is difficult since it requires estimating the costs of heterogeneity compared to its benefits --- in other words, we need not only to know how frequent individuals with a given trait are in the population (this can be answered in our framework), but also to model the individual's \enquote{expense} related to that trait --- and this last point cannot be answered in a universal way, it  depends strongly on the case studied. 

Our method thus consists in comparing the effective trait, which is itself a certain weighted average of the traits, see below~\eqref{def:v}, with three frequently used means, namely geometric, arithmetic and harmonic means, and discuss the influence of heredity and of the traits distribution on their respective evolution. The same method can also apply to compare the effective trait with a model of heterogeneity cost which would imply any other weighted average.

In the following, we apply  our approach to the case of a  population where the individuals are characterised by their size and display heterogeneous growth.

\subsection{Heterogeneous growth rates}

From now on, we consider the population structuring trait as being a \textit{size} variable --- i.e., a trait which grows with time and is partitioned between offspring at division --- and where the heterogeneous trait is the individual growth rate. 
Size is called a \enquote{master trait} for ecological and biodiversity studies~\cite{aksnes2011inherent,Litchman2010}, since for many species --- phytoplankton, marine fauna, bacteria, micro-organisms in general --- it has a positive correlation with many other functional traits.
Size-structured equations appear in many applications, from bacterial growth~\cite{robert2014division,lin2017effects} to polymerisation models~\cite{honore2019growth}. In the case of bacteria,  individual growth rates have been measured and their distribution reveals an important indicator for mutation~\cite{robert2018mutation} or heredity~\cite{delyon2018investigation}. 

We describe the behaviour of such a population by a system of equations satisfied by the concentrations $n_i(t,x)$ of particles of size~$x$ and \emph{type}~$i$ at time $t$. We assume that the individuals all share a common growth rate $\tau(x)$ modulated by a trait $v_i>0$ which depends on their type, so that type $i$ individuals grow at a rate $v_i \tau(x)$. The division rate per unit of size is denoted $\beta(x)$.
When embedded in a time-dependent equation, the division rate \enquote{per unit of size} has to be multiplied by the growth rate, namely $v_i\tau(x)$, to obtain the division rate \enquote{per unit of time}~\cite{taheri-araghi_cell-size_2015,doumic_individual_2023},
namely $v_i\tau(x)\beta(x)$ for type~$i$. This is a pivotal point in our modelling choices: the variable $x$ --- called size, but which could represent any quantity partitioned by division --- is assumed to structure the division.
To guarantee homeostasis at equilibrium, we assume
\begin{equation}\label{as:beta}
  \beta \text{ is integrable at } x=0, \qquad
  \lim_{x\to\infty} x\beta(x)=+\infty.
\end{equation}
Finally, the division of a cell or particle of size $y$ gives rise to two particles of respective size $x$ and $y-x$ with a probability law $b(y,\dd{x})=b(y,y-\dd{x})$. This law is often called in the literature the \textit{fragmentation kernel}. 
The fragmentation kernel $b$ must satisfy $Supp \left(b(y,\cdot)\right)\subset [0,y]$ (the daughters are smaller than the mother) and 
\begin{equation}\label{as:b}
  \int_{0}^y b(y,\dd{x})  =1, \quad
  \int_{0}^y x b(y,\dd{x}) = \frac{y}{2}. 
\end{equation} 
The first equality ensures that $b(y,\dd{x})$ is a probability law,  while the second ensures that the division preserves the overall size --- the sum of the daughter sizes equals the mother size.
We refer to Fig.~\ref{fig:test} for an illustration of the model. 
All these assumptions lead to the following system of equations:
\begin{equation}\label{eq2:SV_cauchy_pb}
  \left\{
    \begin{aligned}
      & \mpdv{t} n_i(t,x)+  v_i \mpdv{x} \big( \tau(x) n_i(t,x) \big)
      = -  v_i \tau(x)  \beta(x) n_i(t,x) \\[-4pt]
      & \quad  + 2 \int_{x}^{\infty} \tau(y) \beta(y) b(y,x)\sum_{j=1}^{M} \kappa_{ji} v_j n_j(t,y)\dd{y},\\[-4pt]
      & \tau(0) n_i(t,0) = 0, \quad n_i(0,x) = n_i^{in} (x).
    \end{aligned}
  \right.
\end{equation}

This system is a particular case of the heterogeneous growth-fragmentation system studied in~\cite{rat2023growth}. In this article, under assumptions on the parameters that we recall in Appendix~\ref{app:assumptions}, it is proven that
there exists a unique Malthus parameter $\lambda=\lambda_{(\boldsymbol{v},\kappa)}>0$, a unique nonnegative steady profile $\boldsymbol{N}={(N_i)}_{1\leq i\leq M}$ and a unique nonnegative adjoint state $\boldsymbol{\phi}={(\phi_i)}_{1\leq i\leq M}$ (the adjoint is obtained by duality when formulating the equation in the sense of distributions) such that  for all $i \in \{1, \ldots, M \}$
\begin{equation}\label{eq2:conv}
  n_i(t,x) e^{-\lambda t} \underset{t\to\infty}{\longrightarrow}
  \rho N_i(x), 
\end{equation}
where the multiplicative factor $\rho$ is expressed as
\begin{equation}\label{eq2:rho}
\rho = \sum_{j=1}^{M} \int_0^{\infty} n_j(0,x)\phi_j(x) \dd{x}.
\end{equation}
The triplet $(\lambda, \boldsymbol{N}, \boldsymbol{\phi})$ is characterised as being the
unique solution to the following system, derived from~\eqref{eq2:SV_cauchy_pb} (and its adjoint equation) by replacing $n_i(t,x)$ by $e^{\lambda t} N_i(x)$:
\begin{align}\label{eq2:N}
  &\left\{
    \begin{aligned}
      &v_i (\tau N_i )'(x) + \lambda N_i(x) = - v_i \tau(x)\beta(x) 
        N_i(x) \\[-4pt]
      &\quad+ 2\int_{x}^{\infty} \tau(y) \beta(y) b(y,x) \sum_{j=1}^{M} \kappa_{ji}v_j N_j(y)\dd{y},\\[-4pt]
      &\tau(0) N_i(0) = 0,
        \quad N_i \geq 0,
        \quad \sum_{j=1}^{M} \int_0^\infty  N_j(s) \dd{s} = 1,
    \end{aligned}
  \right.\\
  \label{eq2:phi}
  &\left\{
    \begin{aligned}
    &-v_i\tau(x)  \phi_i'(x) + \lambda \phi_i(x) = -v_i \tau(x) \beta(x) 
        \phi_i(x) \\[-4pt]
    &\quad + 2 v_i \tau(x) \beta(x) 
      \int_{0}^{x} b(x,y)   \sum_{j=1}^{M} \kappa_{ij} \phi_j (y)\dd{y},\\[-4pt]
    & \phi_i \geq 0, \quad
       \sum_{j=1}^{M} \int_0^\infty N_j(s) \phi_j (s) \dd{s} = 1.
    \end{aligned}
  \right.
\end{align}
To define the effective trait (or effective growth rate)~$v$ introduced in Definition~\ref{def2:ef_fitness}, we consider~\eqref{eq2:N} for $M=1$ and a given $v>0$, namely find $(\lambda_v,N_v,\phi_v)$ solutions to
\begin{align}
  \label{eq:N:hom}
  &\left\{
    \begin{aligned}
    &v (\tau N_v) '(x) + \lambda_v N_v(x) = - v \tau(x) \beta(x) 
        N_v(x) \\[-4pt]
    &\quad + 
        2v\int_{x}^{\infty} \tau(y) \beta(y) b(y,x)   N_v(y)\dd{y},\\[-4pt]
    & \tau(0) N_v(0) = 0,
        \quad N_v \geq 0,
        \quad  \int_0^\infty  N_v(s)\dd{s} = 1,
    \end{aligned}
  \right.\\
  \label{eq:phi:hom}
  &\left\{
    \begin{aligned}
    &-v \tau(x) \phi_v'(x) + \lambda \phi_v(x) = -v\tau(x) \beta(x)
        \phi_v(x)\\[-4pt]
    &\quad+ 2v  \tau(x) \beta(x) 
      \int_{0}^{x} b(x,y)   \phi_v (y) \dd{y} ,\\[-4pt]
    & \phi_v \geq 0, \quad
        \int_0^\infty N_v(s) \phi_v(s) \dd{s}= 1.
    \end{aligned}
  \right.
\end{align}
The effective trait $v$ is such that $\lambda_v=\lambda$; it is uniquely defined thanks to the linearity property of $w \mapsto \lambda_w$, and we have the bounds $v \in [v_1, v_M]$, see Theorem~3.2.\ in~\cite{rat2023growth} and Proposition~\ref{prop:exist} in Appendix. The method described in Section~\ref{subsec:method} consists in comparing $v$ to the distribution of ${(v_i)}_{1\leq i\leq M}$. However,  all the quantities $N_i,\,N_v,\,\lambda,\,\lambda_v$ being defined in an implicit way as solutions to equations~\eqref{eq2:N} or~\eqref{eq:N:hom}, theoretical comparisons in the general case are difficult, so that only numerical investigation has been carried out till now, see~\cite{olivier2017does}. 

We considered two paradigmatic cases for the growth rate.
\begin{description}[leftmargin=0cm]
\item[$\quad$ Case A:\namedlabel{icaseA}{\textcolor{black}{\textbf{A}}}] the growth rate $\tau(x)$ is constant, so that without loss of generality we take $\tau(x)\equiv 1$ and the growth rate for type $i$ is~$v_i>0$. We obtain analytical formulae for constant fragmentation rates $\beta (x)\equiv \beta >0$ and general fragmentation kernel $b(y,x)$:
\begin{equation}\label{caseA}
  \tau(x)\equiv 1, \quad \beta (x)\equiv \beta, \quad b \text{ general}.
\end{equation}

\item[$\quad$ Case B:\namedlabel{icaseB}{\textcolor{black}{\textbf{B}}}] the growth rate depends linearly on the size, so that without loss of generality we take $\tau(x)\equiv x$ and the growth rate for type $i$ is $x\mapsto v_i x$ with $v_i>0$. We derive analytical formulae for general fragmentation rates $\beta (x)$ in the case of a \textit{uniform}  fragmentation kernel $b(y,x)=\f{1}{y}\1_{x\leq y}$. In such a case, the fragmenting particle has a uniform probability to break at any place of the interval $(0,y)$:
\begin{equation}\label{caseB}
  \tau(x)= x, \quad \beta (x)\text{ general}, \quad b(y,x)=\f{1}{y}\1_{x\leq y}.
\end{equation}
\end{description}

\section{Results}

Let us detail the results obtained for size-structured heterogeneous populations reproducing by division. We first derive explicit formulae in Cases~\ref{icaseA} and~\ref{icaseB} defined above, and use them to compare the effective trait with averaged ones. By considering specific heredity kernels, we also determine the (asymptotic) distribution of traits, informative in particular on the most represented trait. Finally, we perform numerical simulations for non-explicit cases, in order to conjecture whether or not the results of the explicit cases can be generalised. 

We recall the expression of the arithmetic, geometric and harmonic means of $\boldsymbol{v}=(v_1, \ldots, v_M)$, respectively:
\begin{equation}\label{def:means}
    \left\{
    \begin{aligned}
    m_A(\boldsymbol{v}) &= \f{v_1 + \ldots + v_M}{M},\\
    m_G(\boldsymbol{v}) &= \sqrt[M]{v_1 \times \ldots \times v_M },\\
    m_H(\boldsymbol{v}) &= \f{M}{\f{1}{v_1} + \ldots + \f{1}{v_M}},
    \end{aligned}
    \right.
\end{equation}
and their comparison: $m_H \leq m_G \leq m_A$.

\subsection{Case A: Constant growth}\label{ssec:A}

Assuming a constant rate of size growth is a reasonable approximation for many applications, e.g.\ \textit{E. coli} in fast growth conditions~\cite{amir2017cell} or yet the mycobacterium tuberculosis~\cite{chung2024single}. It is also one of the most mathematically-studied cases~\cite{perthame_transport_2007}. Constant division rate is a much more specific assumption, however important as it corresponds to a \enquote{neutral} case, where size does not influence the division process.
Assuming~\eqref{caseA},  we can rewrite~\eqref{eq2:N} and~\eqref{eq2:phi}
\begin{align}\label{eq2:N:cst}
  &\left\{
    \begin{aligned}
      &v_i N_i' (x) + \lambda N_i(x) = - v_i \beta
        N_i(x) \\[-4pt]
      &\quad+ 2 \beta \int_{x}^{\infty} b(y,x)  \sum_{j=1}^{M} \kappa_{ji}v_j N_j(y)\dd{y},\\[-4pt]
      &N_i(0) = 0,
        \quad N_i \geq 0,
        \quad \sum_{j=1}^{M} \int_0^\infty  N_j(s)\dd{s} = 1,\\
    \end{aligned}
  \right.\\
  \label{eq2:phi:cst}
  &\left\{
    \begin{aligned}
      &-v_i  \phi_i'(x) + \lambda \phi_i(x) = - v_i \beta 
        \phi_i(x) \\[-4pt]
        &\quad + 2 v_i  \beta
      \int_{0}^{x} b(x,y)   \sum_{j=1}^{M} \kappa_{ij} \phi_j (y)\dd{y},\\[-4pt]
      & \phi_i \geq 0, \quad
       \sum_{j=1}^{M} \int_0^\infty N_j(s) \phi_j(s) \dd{s} = 1.
    \end{aligned}
  \right.
\end{align}
In the homogeneous case $M=1$, a direct computation (where we recall that $\int_0^x b(x,\dd{y})=1$) shows that the adjoint vector $\phi_v \equiv 1$ is constant, and we also have  $\lambda_v=v \beta$~\cite{perthame_exponential_2005}. Moreover, we have explicit solutions for $N_v(x)$ for the two emblematic cases of the fragmentation kernel:

\begin{itemize}[leftmargin=*, parsep=0cm, itemsep=0cm, topsep=0cm]
\item If $b(x,y) =\delta_{y=2x}$ (division into two equally-sized daughters), it has been proven in~\cite[Lemma 2.1]{perthame_exponential_2005} that the triplet solution to~\eqref{eq:N:hom}-\eqref{eq:phi:hom} is given by the following analytical expression:
\begin{equation*}
  \lambda_v= \beta v, \quad \phi_v \equiv 1,
  \quad N_v(x)= C \sum_{n=0}^{+\infty}
  {(-1)}^n\alpha_n e^{-2^{n+1}  \beta x},
\end{equation*}
 with $C$ a normalization constant and 
\begin{equation*}
  \alpha_0 \coloneqq 1,
  \quad \alpha_n \coloneqq \frac{2^n}{(2^n-1)\ldots(2^1-1)},
  \quad   n \geq  1.
\end{equation*}
\item If $b(x,y) = \frac{1}{y}\1_{x\leq y} $ (uniform division), we compute that
the triplet solution to~\eqref{eq:N:hom}-\eqref{eq:phi:hom} is given by
\begin{equation*}
  \lambda_v=\beta v, \qquad \phi_v \equiv 1,
  \qquad N_v(x)= 4 \beta^2 x e^{-2\beta x}.
\end{equation*}
\end{itemize}

The formula $\lambda_v=\beta v$  implies that the effective trait of the heterogeneous case is explicitly given by $v\coloneqq\f{\lambda}{\beta}$. Integrating~\eqref{eq2:N:cst} and summing for all $i$, we also obtain 
\begin{equation}\label{def:v}
  v=\sum\limits_{i=1}^M v_i \int_0^\infty N_i(x)\dd{x}.
\end{equation}
This formula expresses the effective trait as being the average of the traits $v_i$ weighted by $\int_0^\infty N_i \dd{x}$, i.e.\ by the relative amount of the population $i$ in the total population.
It implies that $v \in [v_1, v_M]$ and illustrates well the fact that the effective trait is a certain weighted average of the traits $v_1,\ldots,v_M$; in our case study, it is weighted by the proportion of cells sharing each trait in the population. 

Unfortunately, no similar simple expression relates directly $N_v$ to the $N_i$, as illustrated by~\eqref{def:v}. However, integrating and summing~\eqref{eq2:N:cst}:
\begin{equation}\label{eq:pop_fractions}
  \overline{N}_i
  \coloneqq \int_0^\infty N_i\dd{x}
  = \frac{2}{v + v_i} \sum_{j=1}^{M} \kappa_{ji} v_j \overline{N}_j,
\end{equation}
a formula useful in some specific cases to obtain explicit expressions of $\overline{N}_i$ in terms of $v_i$, $v$ and $\kappa$.
\begin{figure}[H]
  \centering%
   \includegraphics[width=0.4\textwidth]{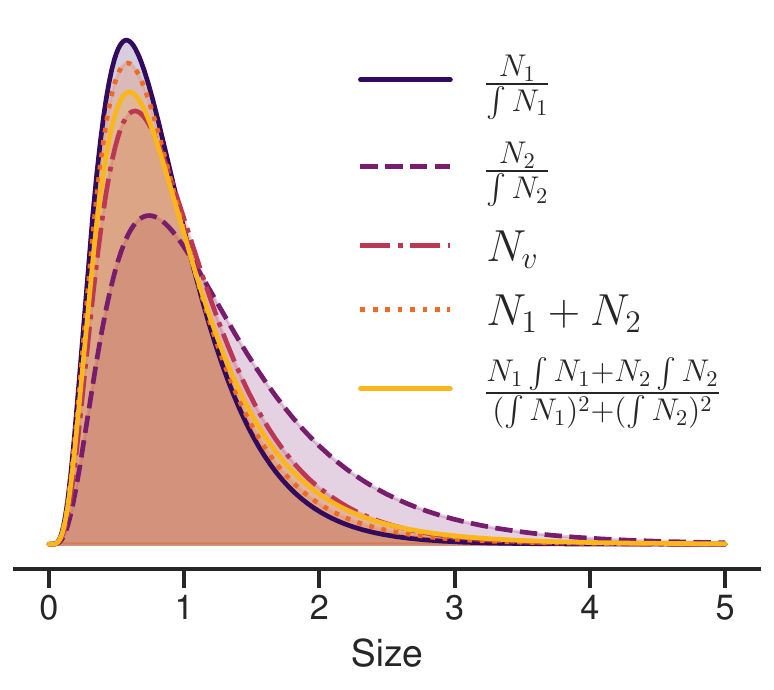}
   \caption{(Case~\ref{icaseA}, $M=2$). Comparison of the homeostatic size distributions {$\boldsymbol{N}=(N_1, N_2)$} in a heterogeneous population, with traits $\boldsymbol{v} = (0.5, 2.5)$ and heredity kernel defined by $(k_1, k_2)= (0.3, 0.5)$, with the size distribution~$N_v$ in a homogeneous population.
   Note that in this case the profile $N_v$ is independent of $v$. 
   All the distributions are 
   normalized to be compared with each other. 
   The last mean of $N_1$ and $N_2$ \emph{(yellow)} corresponds to the mean~$m$ weighted by the $\int N_i$ as defined by~\eqref{def:v}; making clear that formula~\eqref{def:v} between traits ($v=m(\boldsymbol{v})$) does not apply to profiles ($N_v \neq m(\boldsymbol{N})$).
   }\label{fig:N}
\end{figure}

The proof of Proposition~\ref{prop:exist} shows that to have information on the Malthus parameter, we may study either
the direct problem~\eqref{eq:N:hom} or the adjoint one~\eqref{eq:phi:hom}. Explicit solutions in the homogeneous case however suggest that this second option --- looking for solutions to~\eqref{eq2:phi} rather than~\eqref{eq2:N} --- may be easier and thus constitute a better chance to get information on $\lambda$. For this reason, we looked for constant solutions $\phi_i(x)\equiv \phi_i$ which reduces~\eqref{eq2:phi:cst} to a matrix system; this revealed equivalent to integrating~\eqref{eq2:N:cst} and look for $(\lambda, \mathbf{\overline{N}},\mathbf{{\phi}})$,  with $\mathbf{\overline{N}} \coloneqq {(\int_0^\infty \! N_i\dd{x})}_{1\leq i\leq M}$ and \smash{$\mathbf{{\phi}}=(\phi_i)_{1\leq i\leq M}$} constant, positive solution of
\begin{equation*}
  \left\{
  \begin{array}{ll}
    \lambda \mathrm{\overline{N}}_i =-\beta v_i \mathrm{\overline{N}}_i  +2\beta \sum\limits_{j=1}^M \kappa_{ji}v_j  \mathrm{\overline{N}}_j, \quad 1 \leq i \leq M, \\
    { \lambda \phi_i = -\beta v_i \phi_i +2\beta v_i\sum\limits_{j=1}^M \kappa_{ij}  \phi_j, \quad 1 \leq i \leq M. }
  \end{array}
  \right.
\end{equation*} 

Denoting the diagonal matrix $\mathbf{Diag}(\boldsymbol{v})\coloneqq {(v_j \delta_{i,j})}_{1\leq i,j\leq M}$ {and \smash{${A} \coloneqq \beta (-\mathbf{Id}+2\kappa^T )\mathbf{Diag}(v)$}}, we write the system under the form of a matrix equation
\begin{equation}\label{eq:lambda}
  \left\{ \begin{array}{l}
    \lambda \mathbf{\overline{N}}
    = {A} \mathbf{\overline{N}}
    \\
    { \lambda \phi = {A}^T \mathbf{{\phi}}}.
  \end{array}\right.
\end{equation}
Perron-Frobenius theorem applied to $\tilde{A} \coloneqq A+2\beta v_M$
provides us with a unique eigentriplet $(\tilde{\lambda}>0, \mathbf{\overline{N}}>0{, \mathbf{\phi}>0})$, where $\tilde{\lambda}$ is the dominant eigenvalue of the matrix {$\tilde{A}$}. Then $(\lambda= \tilde{\lambda}-2\beta v_M, \mathbf{\overline{N}},\phi)$ is  solution to~\eqref{eq:lambda}, and 
$\lambda >0$ since it satisfies~\eqref{def:v}.
We thus have found a positive solution $(\lambda,\phi)$ to~\eqref{eq2:phi:cst}. By Theorem~\ref{thm:Malthus}, we have uniqueness of a positive eigensolution, so that we can conclude that $\lambda$ is the dominant eigenvalue, i.e.\ the Malthus parameter of the system, and $\phi={(\phi_i)}_{1\leq i\leq M}$ constant is its adjoint eigenvector.
This provides us with an efficient way of computing $\lambda$ numerically, by using for instance the \href{https://numpy.org/doc/stable/reference/generated/numpy.linalg.eig.html#numpy-linalg-eig}{linalg.eig} function of Python's Numpy package, and it enables us to build a tractable characterisation of $\lambda$ as the unique positive root of a polynomial in some particular cases (see Theorem~\ref{thm2:constant}, Appendix).

Let us now detail some particular cases where formulae \eqref{def:v}, \eqref{eq:pop_fractions} and \eqref{eq:lambda} are even more explicit.

\paragraph{Bimodal case} 
In the case of two populations ($M=2$), we define $\kappa$ by
\begin{equation}\label{def:kappa_m2}
  \kappa \coloneqq
  \begin{pmatrix}
    1-{k}_1 & k_1 \\
    k_2 & 1-{k}_2 \\
  \end{pmatrix}.
\end{equation}
with $k_1,\;k_2\in (0,1)$. We obtain the following formula for the effective {trait} (see Proposition~\ref{thm2:M=2} in Appendix for detailed calculations):
\begin{multline}\label{def:lambda:M=2}
  v = \frac{\lambda}{\beta} =(\tfrac{1}{2}-k_1 )v_1 
  + (\tfrac{1}{2}-k_2 ) v_2\\ 
  + \sqrt{{\big( (\tfrac{1}{2}-k_1) v_1 
  - (\tfrac{1}{2}-k_2 ) v_2 \big)}^2 + 4 k_1 k_2 v_1 v_2 }
\end{multline}

The formula appears as a complicated  average of $v_1$ and~$v_2$, which is not so easy to interpret in the general case and highlights that the effective trait --- or equivalently the population fitness --- depends on variability ($v_i$) and heredity/mutation ($\kappa$) in a non-trivial way.

From~\eqref{def:lambda:M=2}, we deduce cases where $\lambda$ superimposes with the geometric and arithmetic means of the traits:
\begin{itemize}[leftmargin=*, parsep=0cm, itemsep=0cm, topsep=0cm]
  \item If $\kappa$ is uniform,
i.e.\ $k_1=k_2=\frac{1}{2}$, the effective {trait} is the
{geometric mean} of the traits:
\begin{equation*}
  v = m_G(v_1, v_2) = \sqrt{ v_1 v_2 }.
\end{equation*}
\item If 
$k_1 = k_2 = \frac{1}{4}$, the effective {trait} is the
{arithmetic mean} of the traits:
\begin{equation*}
  v = m_A(v_1, v_2) = \mfrac{v_1 + v_2}{2}.
\end{equation*}
\end{itemize}
We illustrate this result on Fig.~\ref{fig:v_eff_M2_specific}, and we will see for the general case how to interpret the case $k_1=k_2=\f{1}{4}$.

Testing other types of kernel indicates that none of the classical means (arithmetic, geometric and harmonic) is, for every kernel~$\kappa$ and every set of traits, a bound of the effective {trait}, see Fig.~\ref{fig:v_eff_M2_v_mean}. 
In fact, the bounds $v\in (v_1,v_2)$ are optimal. Assuming $v_1<v_2$, we have that phenotype~2 having the highest growth rate tends to dominate when strongly inherited (i.e.\ if individuals having trait~$v_2$ almost always give birth to individuals of the same trait): we observe in~\eqref{def:lambda:M=2} that if $k_2\to 0$ (perfect heredity of $v_2$), then $v\to v_2$ for any value of $k_1$. This is not sufficient to have reciprocally $v\to v_1$ if $k_1\to 0$: indeed, if $k_1\to 0$ we have $v\to \max(v_1,(1-2k_2)v_2)$, so that $v\to v_1$ if and only if $v_1\geq (1-2k_2)v_2$.
Fig.~\ref{heatmap} illustrates this behaviour.

\begin{figure}[H]
  \centering%
  \includegraphics[width=0.33\textwidth]{%
    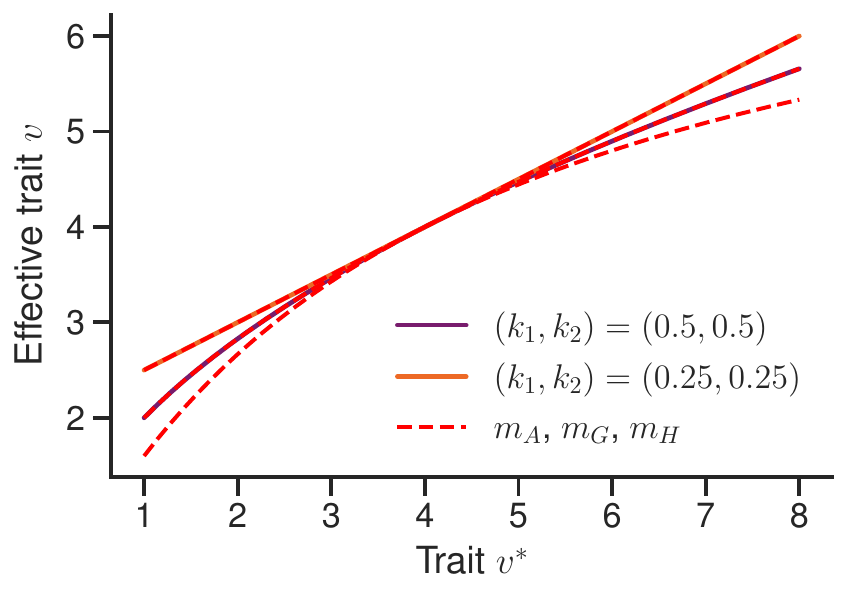}
  \caption{(Case~\ref{icaseA}, $M=2$). Variation of the effective trait v of a population with two traits: one 
  fixed, equal to $4$, and the other one, $v^*$, varying in $[1, 8]$, for different $\kappa$ (\emph{full lines}).~\emph{Red dashed lines} correspond to the arithmetic, geometric and harmonic means of $4$ and $v^*$, with $m_H \leq m_G \leq m_A$. When $\kappa$ is uniform 
  \emph{(purple)}, {$v$} 
  coincides with $m_G(4, v^*)$;
  when $k_1=k_2=0.25$ \emph{(orange)}, it coincides with $m_A(4, v^*)$.}\label{fig:v_eff_M2_specific}
\end{figure}
\begin{figure}[H]
  \centering%
  \includegraphics[width=0.49\textwidth]{%
    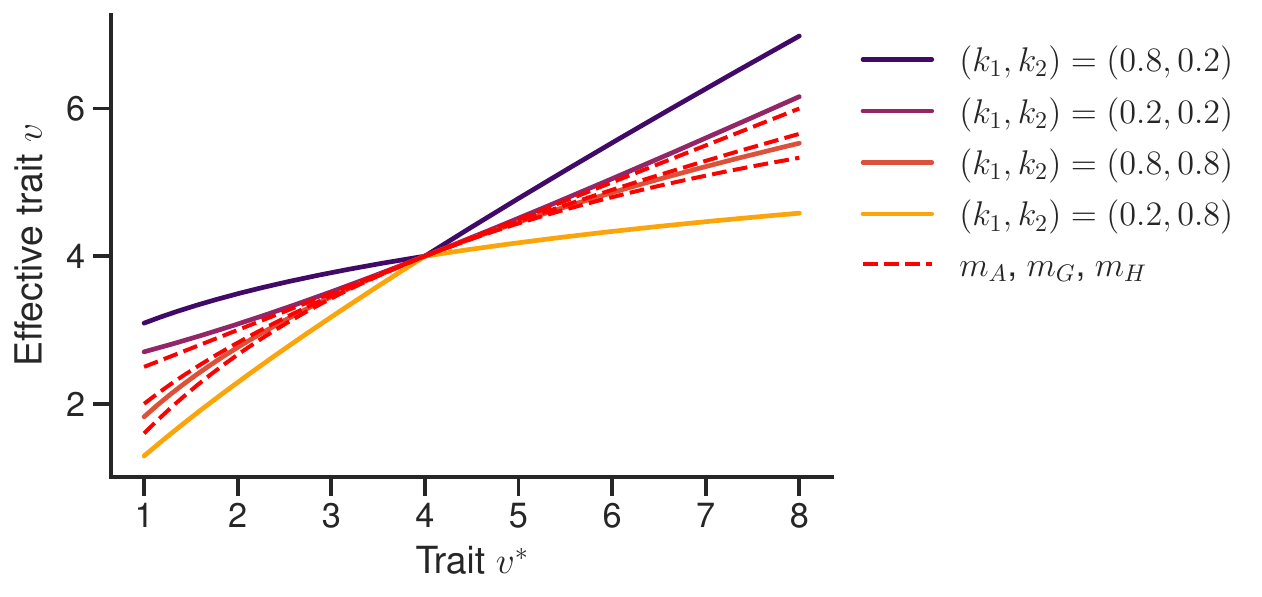}
  \caption{(Case~\ref{icaseA}, $M=2$). As in Fig.~\ref{fig:v_eff_M2_specific}, two traits $(v_1, v_2)=(\min(4, v^*), \max(4, v^*))$
  with 
  $v^*$ varying in $[1, 8]$, and for different kernels $\kappa$ (\emph{full lines}).
  The cases $k_1=k_2=0.2$ \emph{(purple)} or $0,8$ \emph{(orange)} display symmetric roles for the two traits: strong self-reproducing heredity (less mixture, {$k_1=k_2=0.2$}) versus favoring the other phenotype (more mixture, $k_1=k_2=0.8$).
  In the other cases, one trait is more likely to be given at birth: $v_1$ for $(k_1, k_2)=(0.2, 0.8)$ \emph{(yellow)}, and $v_2$ for $(0.8, 0.2)$ \emph{(blue)}, resulting in a decrease, respectively increase, of the effective trait. For some kernels, the effective trait $v$ can be smaller or larger than any classical mean of the traits \emph{(dotted lines)}. In particular, when heredity is strong for individuals of type~$2$ ($k_2$ is small, i.e.\ $\kappa_{22} = 1-k_2$ is large), $v$ 
  is larger than the standard means.}\label{fig:v_eff_M2_v_mean}
\end{figure}

\noindent The asymmetry in the behaviour of $v$ between the limits $k_1 \to 0$ and $k_2 \to 0$ is better understood by looking at the population fractions $\overline{N}_i = \int N_i$, using $v= v_1\overline{N}_1 + v_2\overline{N}_2$ from~\eqref{def:v}. Formula~\eqref{eq:pop_fractions} gives
\begin{equation*}
  \overline{N}_1 
  = \frac{2 k_2 v_2}{v - \tilde{v}_1} \overline{N}_2 
  = \frac{v - \tilde{v}_2}{2 k_1 v_1} \overline{N}_2,
  \quad \tilde{v}_i = (1- 2 k_i) v_i,
\end{equation*}
with $v-\tilde{v}_i$ positive by positivity of the $\overline{N}_i$. Formally, as $k_2 \to 0$ we get $\overline{N}_1 \to 0$, and $v \to v_2$ (by~\eqref{def:v} or $\tilde{v}_2 \to v_2$): trait~$v_1$ is overtaken in the long-time, regardless of~$k_1$.

\newpage
\noindent Likewise, letting $k_1 \to 0$, we recover the two scenarios:
\begin{itemize}[leftmargin=*, parsep=0cm, itemsep=0cm, topsep=0cm]
\item If $\tilde{v_2} \leq v_1$, then $v - \tilde{v}_2 > 0$ for any $k_1$. Taking $k_1$ to zero leads to $N_2 \to 0$, hence $v \to v_1$: trait $v_2$ becomes negligible.
\item Otherwise, $v-\tilde{v}_2$ cannot stay positive if $v \to v_1$, so necessarily $v \to \tilde{v_2}$ as $k_1 \to 0$. In the limit, $v > v_1$, and because of~\eqref{def:v} $N_2 \to 0$ cannot hold: both traits coexist.
\end{itemize}

\begin{figure}[H]
  \centering%
  \includegraphics[width=0.4\textwidth]{
  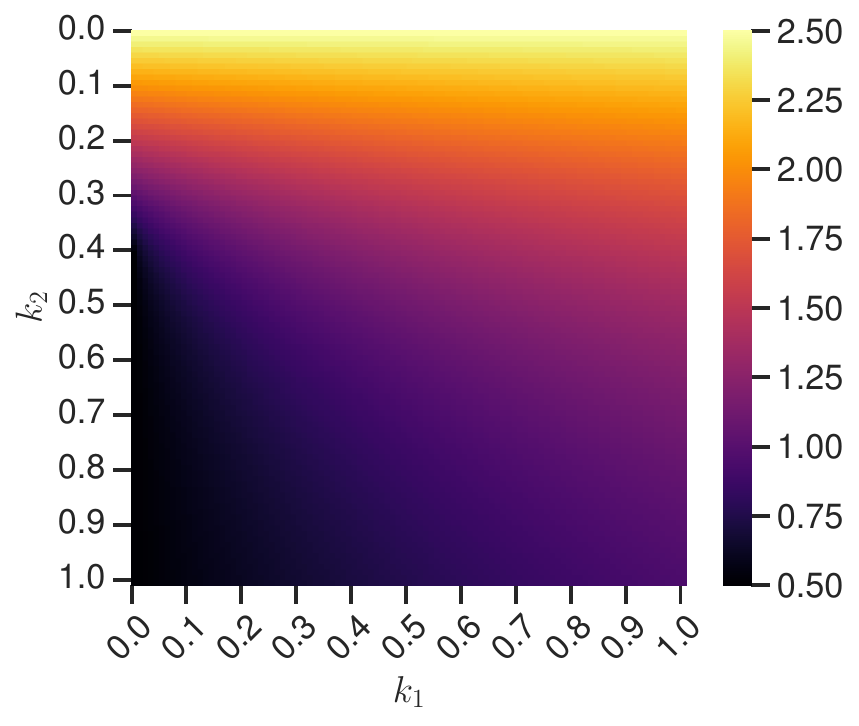}
  \caption{(Case~\ref{icaseA}, $M=2$). Effective trait $v$ of a heterogeneous population with traits $\boldsymbol{v}=(0.5, 2.5)$ with respect to the coefficients $k_1=\kappa_{12}$ and $k_2=\kappa_{21}$ of the heredity kernel $\kappa$. 
  The map is not symmetric and we retrieve that $v \to v_2$ when $k_2 \to 0$ while $v$ does not necessarily go to $v_1$ when $k_1 \to 0$.}\label{heatmap}
\end{figure}
\vspace{-.8cm}
\begin{figure}[H]
    \centering
    \includegraphics[width=0.31\textwidth]{%
    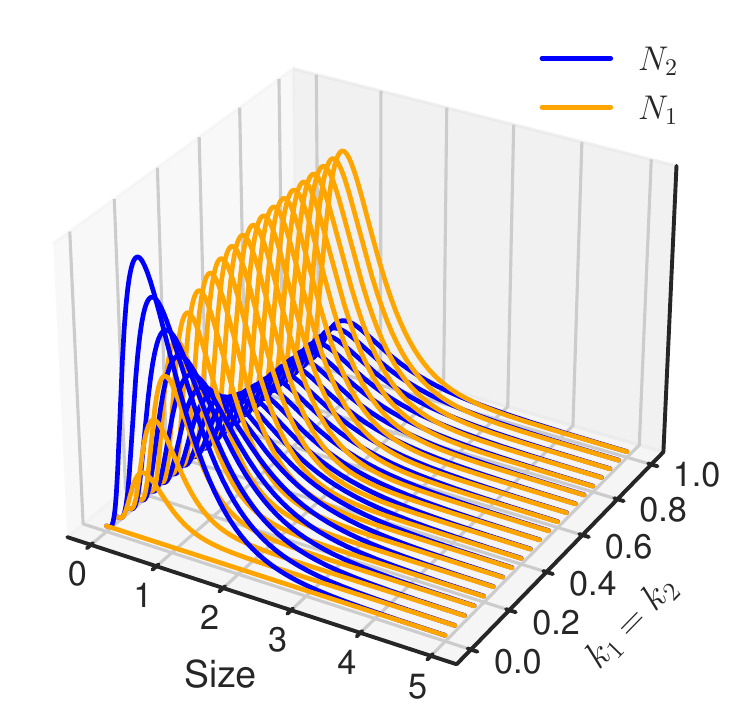}
    \vspace{-.1cm}
    \includegraphics[width=0.31\textwidth]{%
    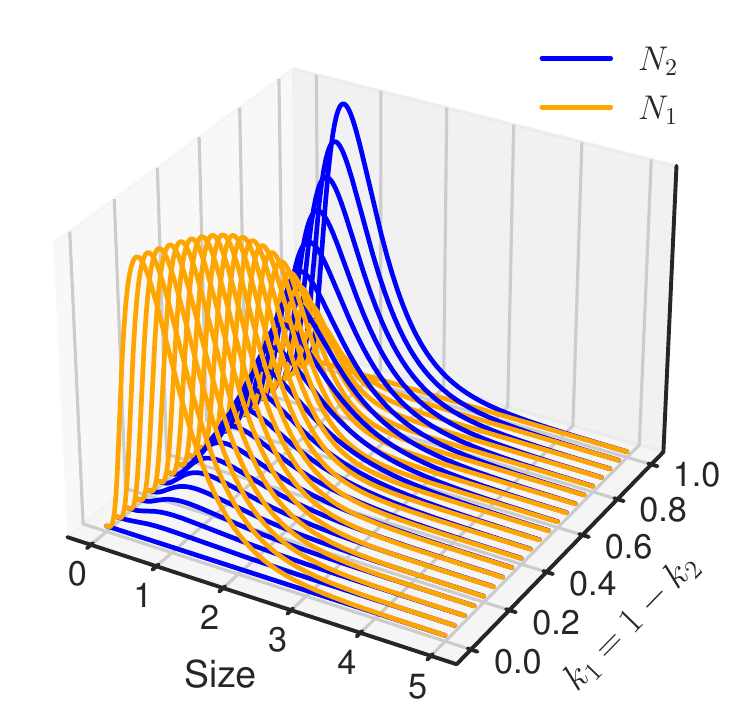}
    \caption{(Case~\ref{icaseA}, $M=2$). Homeostatic size distributions {$\boldsymbol{N}=(N_1, N_2)$} in a heterogeneous population with traits $\boldsymbol{v} = (0.5, 2.5)$ and heredity kernel defined by $k_1=k_2$ (\emph{top}) or $k_1=1-k_2$ (\emph{bottom}), with respect to $k_1$.}\label{fig:3D}
\end{figure}

Thus in the limit, one subpopulation gets outcompeted (or disappears if $k_i > \frac{1}{2}$), i.e.\ $\overline{N}_i \to 0$ and $v \to v_i$, except when subpopulation $1$, with slower growth rate, is self-reproducing (i.e.\ $k_1 \to 0$) and subpopulation $2$, even though continuously \enquote{leaking} into subpopulation $1$ ($k_2 \neq 0$), reproduces fast enough in comparison (namely $\tilde{v}_2 = (1-2 k_k)v_2 \geq v_1$).
In this case, the two subpopulations coexist, making the effective trait converging to an intermediate value $v \to \tilde{v}_2 \in (v_1, v_2)$.
For further intuition, we refer to~\cite[Section 3.3]{rat2023growth} which, although considering different coefficients, focuses on the case $M=2$ and $\kappa$ such that either $k_1$ or $k_2$ is zero.

In order to better illustrate how the weights $(\int N_1,\int N_2)$ are influenced by $\kappa$, thus giving rise to an effective trait closer to $v_1$ or $v_2$, we plotted in Figs.~\ref{fig:3D} and~\ref{figS6}, a 3D visualisation of $N_1$ and $N_2$ as functions of $x$ and  of a parameter of the kernel~$\kappa$. As expected, we see that the faster subpopulation dominates when its trait $v_2$ is strongly inherited ($k_2$ small), which is not always the case for the slower subpopulation.
%

\paragraph{General multimodal case}

For the general multimodal case, we may solve numerically~\eqref{eq:lambda} to obtain the dominant eigenvalue $\lambda$, but there is no fully explicit formula as for $M=2$.
However, explicit formulae can be obtained for some interesting heredity kernels $\kappa$.
\begin{itemize}[align=right,itemindent=2em,labelsep=4pt,labelwidth=1em,leftmargin=0pt,nosep, itemsep=1pt, topsep=1pt]
\item No heredity: $\kappa_{ij} = \kappa_j$ (the distribution of the daughter traits is independent to the mother trait). We express  $\lambda$ as the unique positive root of a polynomial (see Theorem~\ref{thm2:constant} and its proof in Appendix for more details) and the population fractions, thanks to~\eqref{eq:pop_fractions} and~\eqref{def:v}, as:
\begin{equation} \label{eq:popfrac_noheredity}
  \int N_i = \frac{2 v \kappa_i}{v + v_i}.
\end{equation}
Although not fully explicit because of $v$, this equilibrium well reflects a trade-off between heredity/mutation and  individual advantage/fitness. On the one hand, increasing $\kappa_i$ raises the proportion of individuals of trait $v_i$ by making it more likely at birth; on the other hand, a higher $v_i$ means individuals more likely to divide earlier and thus disappearing faster, reducing their long-time contribution.

\item Case $\kappa_{ii}(\alpha) = \alpha$ and $\kappa_{ij}(\alpha)=\frac{1-\alpha}{M-1}$, for $\alpha \in [0, 1)$. We still do not have an explicit formula for general $\alpha$, see Section~\ref{ssec:beneficial} for in-depth numerical exploration and the rationale behind this definition. Again,~\eqref{eq:pop_fractions} and~\eqref{def:v} yields:
\begin{equation}\label{eq:popfrac_alpha}
  \int N_i
  = \frac{\frac{1-\alpha}{M} v}{ \frac{M-1}{2M} v 
  + (\frac{M+1}{2M} - \alpha) v_i}
\end{equation}
Interestingly, the population fraction $\int N_i$ either increases or decreases with $i$, depending on whether $\alpha$ is greater or smaller than the threshold $\alpha_0 = \frac{1}{2} + \frac{1}{2M}$.

\item Uniform kernel: $\kappa_{ij}=\f{1}{M}$ (a particular case of the two previous).
Unlike the case $M = 2$, the effective trait is no longer given by the geometric mean when $M > 2$ (see Fig.~\ref{fig:influence_of_M}).

Furthermore, if traits are drawn uniformly at birth, population fractions are not equal: \eqref{eq:popfrac_noheredity} or~\eqref{eq:popfrac_alpha} yields
\begin{equation}\label{eq:popfrac_uniform}
  \int N_i =  \frac{2}{M} \frac{v}{v + v_i},
\end{equation}
This corresponds to a balance at equilibrium in favor of the slow-dividing individuals: the individuals with higher growth rate divide more often, and when dividing give rise to all traits uniformly. Therefore, they disappear more from the population than the individuals with lower growth.

\item Case $\kappa(\alpha_0)$: $\kappa_{ii}=\f{1}{2}+\f{1}{2M}$ and $\kappa_{ij}=\f{1}{2M}$. In this case however, we can prove that the effective trait, as for $M=2$, is equal to the arithmetic mean. To do so, we notice that
\begin{equation*}
  v = \f{1}{M}\sum\limits_{j=1}^M v_j, 
  \qquad \phi_i
  = \f{v_i}{\frac{1}{M} \sum\limits_{j=1}^{M} v_j}
  = \frac{v_i}{v}
\end{equation*}
is a solution to~\eqref{eq2:phi:cst}, and conclude by uniqueness.  
We also deduce from~\eqref{eq:popfrac_alpha} that
\begin{equation*}
  \int N_i 
  = \f{1}{M}.
\end{equation*}
This case may be interpreted as one daughter keeping the trait of its mother, whereas the other daughter picks its own uniformly among all traits: then division creates all traits in a uniform manner, irrespective to the (trait-dependent) division rate, while keeping the amount of the dividing phenotype unchanged. This may be considered as \enquote{more uniform than the uniform kernel}, in the sense that the uniform kernel fails to produce a similar uniform distribution of traits in the population at equilibrium, as seen explicitly with~\eqref{eq:popfrac_uniform}.
\end{itemize}

\subsection{Case B: linear growth rate}

Many bacteria, among other species, display a linear growth rate $\tau(x)=x$. In the homogeneous case, a simple integration of~\eqref{eq:N:hom} multiplied by $x$ shows that
\begin{equation*}
  \lambda_v=v,
\end{equation*}
i.e.\  the individuals grow exponentially at the same rate as the Malthusian parameter characterising the growth of the population. Moreover, the adjoint function is 
\begin{equation*}
  \phi_v (x) = C x.
\end{equation*}
This remarkable fact remains true for any fragmentation kernel and rate, as soon as balance assumptions between coefficients ensure homeostasis. Unfortunately, these properties are not satisfied in general for heterogeneous populations. We thus restrict ourselves to the uniform division kernel $b(y,x)=\f{1}{y}\1_{x\leq y}$.
Systems~\eqref{eq2:N}--\eqref{eq2:phi} become
\begin{align}\label{eq:N:lin}
  &\left\{
    \begin{aligned}
      &v_i \big(xN_i (x)\big)' + \lambda N_i(x) = - v_i x\beta(x) 
        N_i(x) \\[-4pt]
     &\quad+ 2\sum_{j=1}^{M} \kappa_{ji}v_j \int_{x}^{\infty}  \beta(y)   N_j(y)\dd{y},\\[-4pt]
      & N_i \geq 0,
        \quad \sum_{j=1}^{M} \int_0^\infty  N_j(x)\dd{x} = 1,\\
    \end{aligned}
    \right.
\end{align}
\begin{align}\label{eq:phi:lin}
  &\left\{
    \begin{aligned}
      &-v_i x  \phi_i'(x) + \lambda \phi_i(x) = -v_i x\beta(x)  
        \phi_i(x) \\[-4pt]
        &\quad + 2 v_i  \beta(x) 
       \sum_{j=1}^{M} \kappa_{ij} \int_{0}^{x}\phi_j (y)\dd{y},\\[-4pt]
      & \phi_i \geq 0, \quad
       \sum_{j=1}^{M} \int_0^\infty N_j (s) \phi_j (s) \dd{x} = 1.
    \end{aligned} 
  \right.
\end{align}
In~\cite{doumic_jauffret_eigenelements_2010}, an explicit solution for the homogeneous case with $\beta(x)=\beta x^{n-1}$ with $n \geq 1$ is given by
\begin{equation*}
  N_v(x)= C e^{-\f{\beta}{n}x^{n}}.
\end{equation*}
We notice that this formula remains true for any division rate $\beta(x)$, for which we obtain, for a certain normalization constant $C$
\begin{equation*}
  N_v(x)= C e^{-\int_0^x \beta(s)\dd{s}}.
\end{equation*}
We thus look for solutions $N_i(x)$ of the form
\begin{equation}\label{def:Niconstant}
    { N_i(x)\coloneqq C \left(\int N_i \right)\, e^{-\int_0^x \beta(s) \dd{s}},}
\end{equation}
with $C={\left(\int e^{-\int_0^x \beta(s) \dd{s}} \dd{x}\right)}^{-1}$.
Denoting the effective trait $v=\lambda$, by plugging~\eqref{def:Niconstant} into~\eqref{eq:N:lin}  and dividing by $C {e^{-\int_0^x \beta(s) \dd{s}}}$, we obtain the system
\begin{equation}\label{effective:v:lingrowth}
  v { \int N_i} 
  =-v_i  \int N_i+ 2\sum\limits_{j=1}^M \kappa_{ji}v_j { \int N_j}. 
\end{equation}
This is exactly the matrix equation~\eqref{eq:lambda} in the case $\beta=1$, so that all the same conclusions as for Case~\ref{icaseA} hold. We can thus write the effective trait~$v$ as a weighted average of the traits $v_i$: summing all the equations~\eqref{effective:v:lingrowth},  we obtain once again~\eqref{def:v}.
Note that whereas the shape of the solution is preserved for the eigenfunctions $N_i$, the adjoint eigenfunctions~$\phi_i$  are no longer defined by linear functionals.

\subsection{Is heterogeneity detrimental or beneficial?}\label{ssec:beneficial}

In all of the previous analyses, we have shown that the question of whether heterogeneity is detrimental or beneficial cannot be answered rigorously without deciding which is the correct average to model the costs associated with  each trait. However, let us explore here how the number of traits and their variance can affect the effective trait of the population.

\paragraph{Influence of the number of traits} 
We fix the interval $\mathcal{V}$ of the traits and we test the influence of the number $M$ of traits on the value of the effective trait. 
In the framework of Cases~\ref{icaseA} or~\ref{icaseB}, for various $M$, we compute an approximation of the effective trait~$v^M$  for a given kernel $\kappa_{_M}$ when the set of traits is uniformly distributed over $\mathcal{V}$.
Fig.~\ref{fig:influence_of_M} shows that the effective {trait} converges to some value $v^{\infty}$ as $M$ grows, even when the kernels ${(\kappa_{_M})}_{M \geq 2}$ are (independent identically distributed) random matrices.
Such result would constitutes a key step to define the effective fitness in a continuous-trait setting.

\begin{figure}[H]
\centering
  \includegraphics[width=0.4\textwidth]{%
  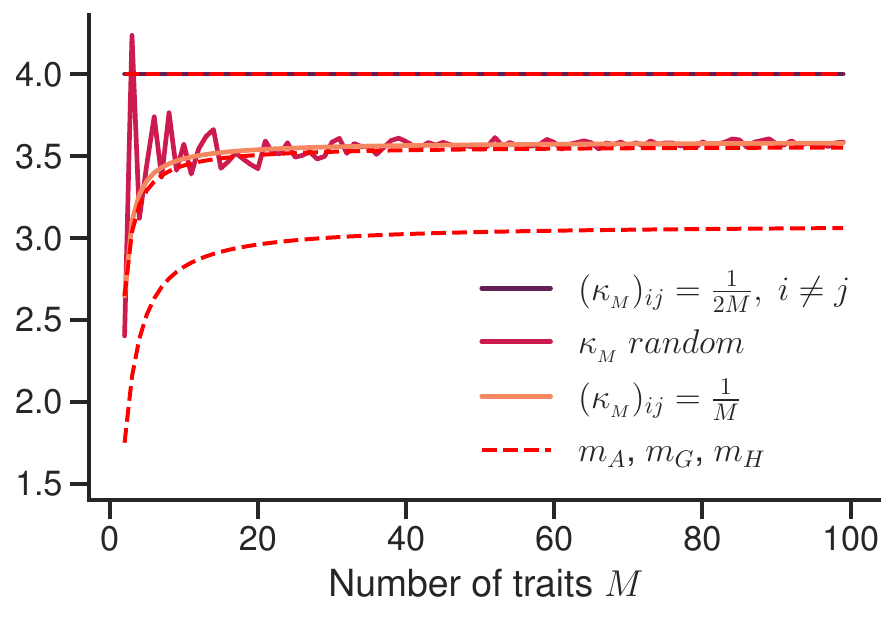}
  \caption{(Case~\ref{icaseA}, $M>2$). Variation of the effective trait of a population defined by $M$ traits regularly spaced in~$\cV = [1, 7]$ and the kernel~$\kappa_{_M}$, with respect to~$M$.
  The random $\kappa_M$ \emph{(pink line)} are generated by independently sampling each entry from a uniform distribution on $[0,1]$, then normalizing each row to ensure the matrix is stochastic.
  \emph{Red dashed lines} correspond to the arithmetic, geometric and harmonic means of the traits, with $m_H \leq m_G \leq m_A$.}\label{fig:influence_of_M}
\end{figure}

\paragraph{Varying the variability and the mother-daughter correlation of the traits for fixed number of traits $M$ and fixed mean trait $\bar{v}$} 
We test the influence of the variance between traits by fixing the number $M$ of traits and considering the sets of traits $\boldsymbol{v^\sigma}=(v^\sigma_1, \ldots, v^\sigma)_{_M}$, for several values of~$\sigma$, distributed over the interval~$\mathcal{V}_{\sigma}$ such that:
\begin{equation}\label{trait_distribution}
  \abs{\cV_\sigma} = v_M^\sigma - v_1^\sigma = \sigma, \quad 
  m(\boldsymbol{v}^\sigma) = \bar{v}, \quad 
  \forall \sigma > 0,
\end{equation}
for $m$ one of the means defined by~\eqref{def:means}.
When $m=m_A$ is the arithmetic mean, the $v^\sigma_i$ are distributed uniformly in $\cV_\sigma$, while for $m=m_G$ and $m_H$ this is the $\log(v_i^\sigma)$ and $\frac{1}{v_i}$ which are equally spaced, respectively. 
Note that for any weighted mean $m$, we have $\boldsymbol{v}^\sigma = \sigma \boldsymbol{v} + (1-\sigma) \bar{v}$, which yields $\Var(\boldsymbol{v}^\sigma) = \sigma  \Var (\boldsymbol{v})$, where $\Var(\boldsymbol{\omega}) \coloneqq m(\boldsymbol{\omega}^2) -m(\boldsymbol{\omega})^2$ denotes the variance of a set of traits $\boldsymbol{\omega}$. In this case, $\sigma$ directly modulates the variance, or equivalently the coefficient of variation $CV_\sigma = \sigma CV_1 \coloneqq \Var(\boldsymbol{v}) / \bar{v}$.

It was observed by Olivier~\cite{olivier2017does} that for a  linear growth rate, in
the absence of heredity and for $\Bar{v}$ defined as the mean trait
at birth 
\begin{equation*}
  \Bar{v} = \sum_{i=1}^{M} v_i \kappa_i,
\end{equation*}
reducing the variance $\sigma$ among individual traits enhances the overall growth of the population (i.e.\ increases the Malthus parameter, or equivalently the effective trait). The same type of result is stated by Lin and Amir~\cite{lin2017effects}: they besides consider the case of heredity with positive correlations (reported e.g.\ in \emph{E. coli}) reporting that the variability in growth rate is indeed beneficial to population growth providing that the mother and daughter cells' growth rates are (positively) correlated strongly enough.
This may be directly observed on the formula~\eqref{def:lambda:M=2}: if $k_1,~k_2\to 0$, $v$ tends to $v_2$, the largest possible value for~$v$. Actually, it is even sufficient to have strong heredity of the highest trait only: $k_2 \to 0$ implies $v \to v_2$, as seen in Section~\ref{ssec:A}. Similarly, in the general case, we can see on~\eqref{effective:v:lingrowth} that if $\kappa_{Mi}\to \delta_{i=M}$, we have $v\to v_M$ and $\f{N_i}{N_M}\to 0$ for $i<M$ (see Figs.~\ref{fig:3D} and~\ref{figS6}). The population with higher growth rate reproduces more frequently and, by often producing offspring of the same type, eventually dominates the other subpopulations.

To specifically assess the influence of the mother-daughter correlation in growth rates in how variability affects population growth, we consider kernels of the form
\begin{equation}\label{eq2:def_kappa_almost_uniform}
  \kappa_{ij} (\alpha) = 
  \begin{cases}
    \alpha, \quad &\text{if~~} i = j,\\
    \frac{1-\alpha}{M-1}, \quad &\text{if~~} i \neq j,
  \end{cases}, \qquad
  \forall (i, j) \in \{1, \ldots, M\},
\end{equation}
for different values of $\alpha \in [0, 1)$. Indeed, $\alpha$ provides a good parameterization of the Pearson correlation between mother and daughter traits, as proved in Appendix~\ref{sec:cor}. Specifically:
\begin{itemize}[leftmargin=*, parsep=0cm, itemsep=0cm, topsep=0cm]
\item $\alpha = 0$ reflects strong negative correlation, as it results in a daughter trait different from the mother's trait.

\item $\alpha = \frac{1}{M}$ indicates to no correlation, as the trait distribution at birth becomes (uniform) independent of the mother's trait.

\item $\alpha=1$ corresponds to perfect positive correlation, where the daughter always inherits the mother's trait.
\end{itemize}

We extend the results of Olivier~\cite{olivier2017does} and {Lin and} Amir~\cite{lin2017effects} (linear growth rate and equal mitosis) to the Case~\ref{icaseA} (constant coefficients, Fig.~\ref{fig:v_eff_wrt_std_correlation}). In Supplementary Information, we simulated the case of equal mitosis with linear growth rate and constant or linear $\beta$.

The numerical results, plotted on Figs.~\ref{fig:v_eff_wrt_std_correlation},~\ref{fig:v_eff_wrt_std} (and~\ref{fig:v_eff_wrt_std_correlation_all}) are the following:
\begin{itemize}[leftmargin=*, parsep=0cm, itemsep=0cm, topsep=0cm]
\item When the range $\sigma$ of the set of traits is fixed, the effective trait $v_\alpha$ increases with $\alpha$, i.e.\ with the mother-daughter correlations of traits (or growth rates). As detailed in Section~\ref{ssec:A}, this can be understood via $v= \sum v_i \overline{N}_i$ (from \eqref{def:v}) and an $\alpha$-dependent comparison of the population fractions $\overline{N}_i$ by \eqref{eq:popfrac_alpha}.

\item Cell-to-cell variability in growth rate, {when leaving the arithmetic mean unchanged}, either decreases or increases population growth depending on whether or not~$\alpha$ is lower or greater than a threshold~$\alpha_0$.
Numerical simulations on Fig.~\ref{fig:v_eff_wrt_std} illustrate that for $M>2$, the threshold value is $\alpha_0 \coloneqq \frac{1}{2} + \frac{1}{2M}$ and $v_{\alpha_0}=m_A(\boldsymbol{v})$ the arithmetic mean.
\end{itemize}

\end{multicols}
\begin{figure}[H]
  \includegraphics[height=4.1cm, trim={0cm 0 5.7cm 0}, clip=true]{%
    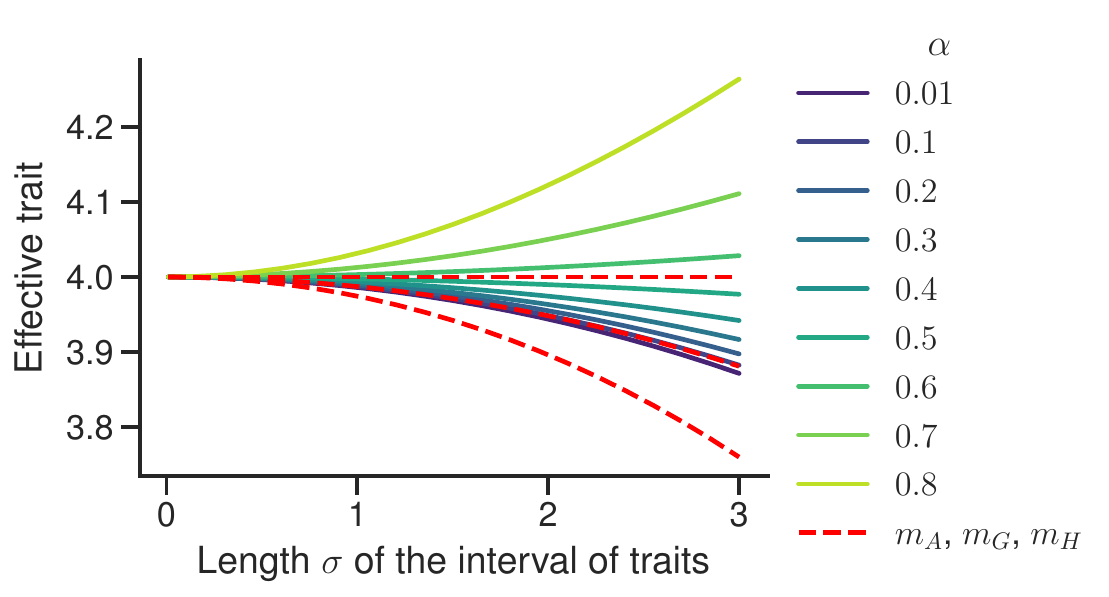}
  \includegraphics[height=4.1cm, trim={1cm 0 5.7cm 0}, clip=true]{%
    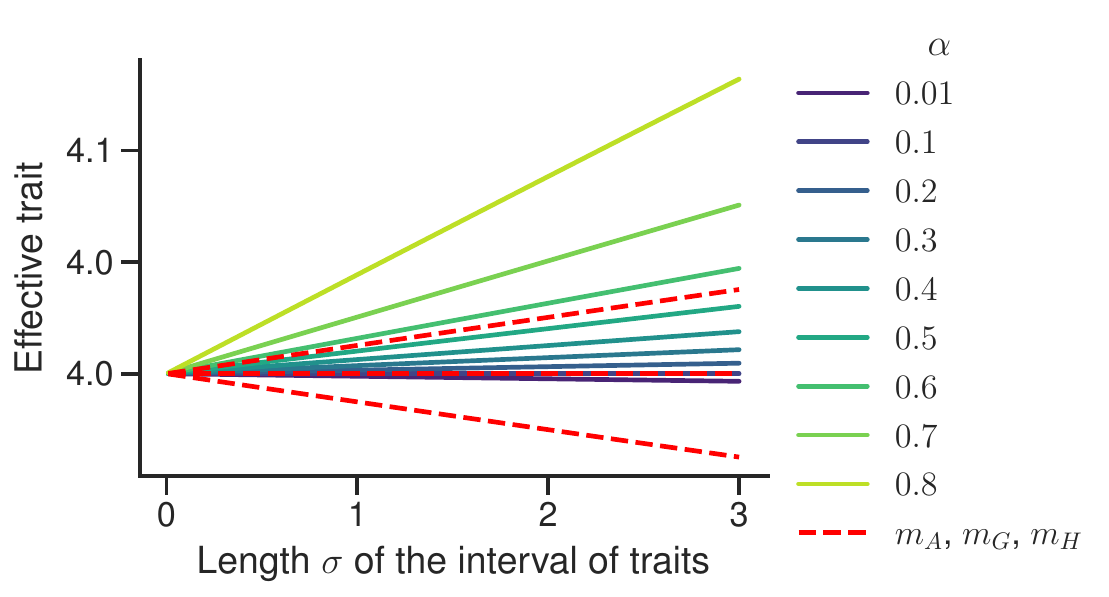}
  ~\includegraphics[height=4.1cm, trim={.9cm 0 .1cm 0}, clip=true]{%
    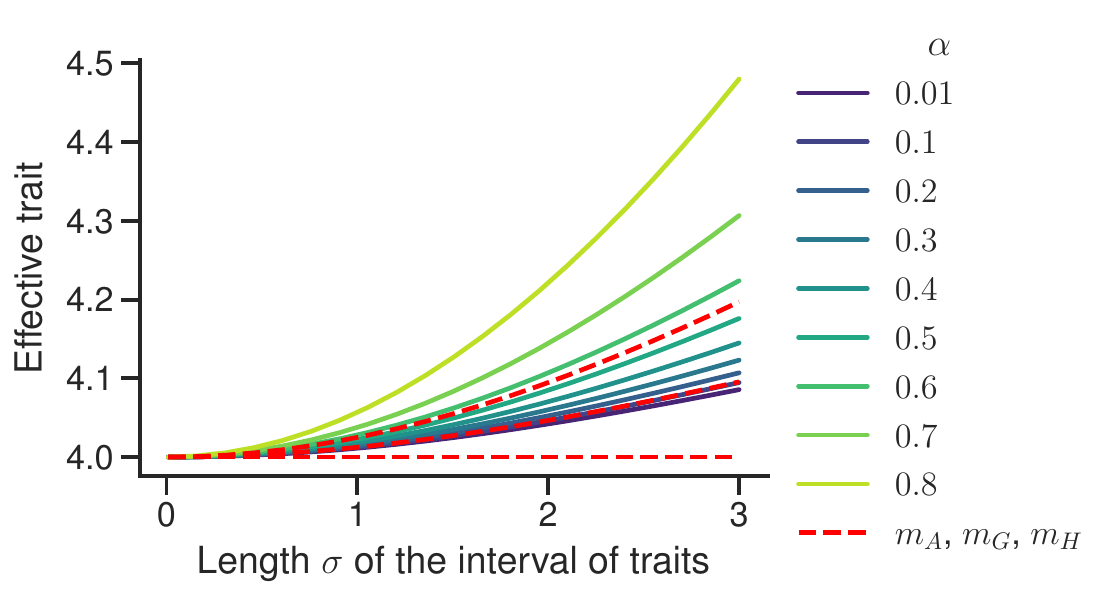}
  \caption{(Case~\ref{icaseA}, $M=10$).
   Variation of the  effective trait of a population of $M=10$ traits and heredity kernel~$\kappa(\alpha)$, defined by~\eqref{eq2:def_kappa_almost_uniform}, with respect to the range~$\sigma$ of the~traits (x-axis) and $\alpha$ (line color). Individual traits are distributed according to~\eqref{trait_distribution}: their $m$-mean is $\Bar{v}=4$, with $m$ being either $m_A$ \emph{(left)}, $m_G$ \emph{(middle)}, or $m_H$ \emph{(right)}, while spanning a (varying) range~$\sigma$. Parameter $\alpha$ increases with the mother-daughter Pearson correlation of traits which cancels for 
    $\alpha = \frac{1}{M}= 0.1$.}\label{fig:v_eff_wrt_std_correlation}
\end{figure}

\begin{figure}[H]
  \centering
  \includegraphics[height=4.38cm, trim={0cm 0 7.2cm 0}, clip=true]{%
    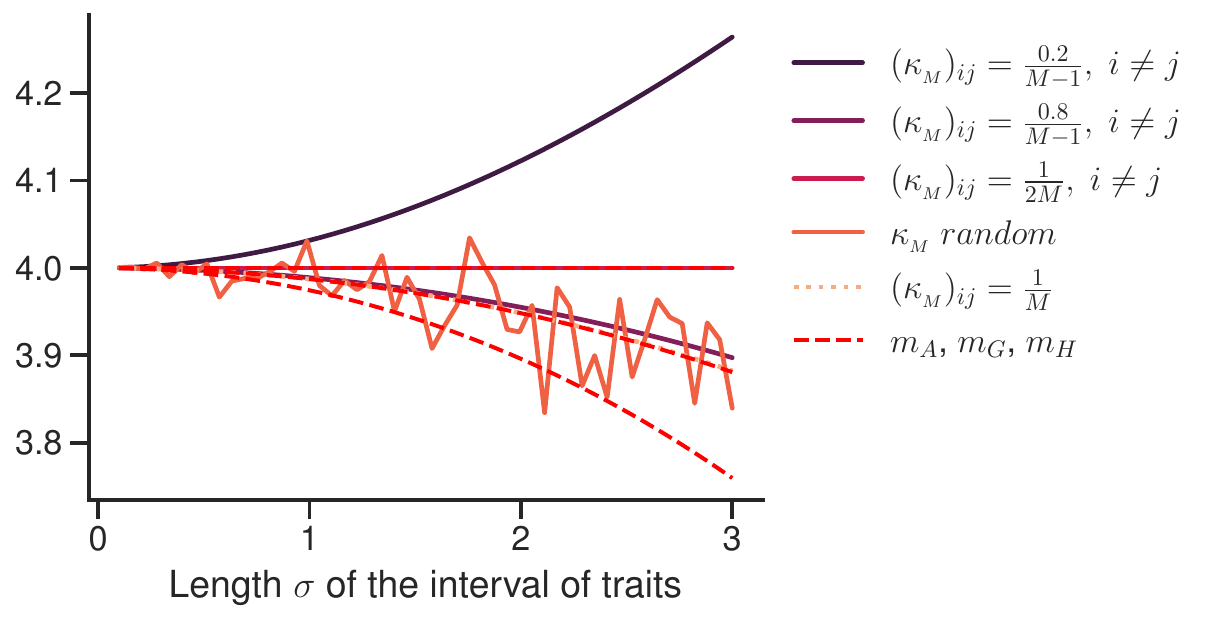}
  ~~~
  \includegraphics[height=4.38cm]{%
    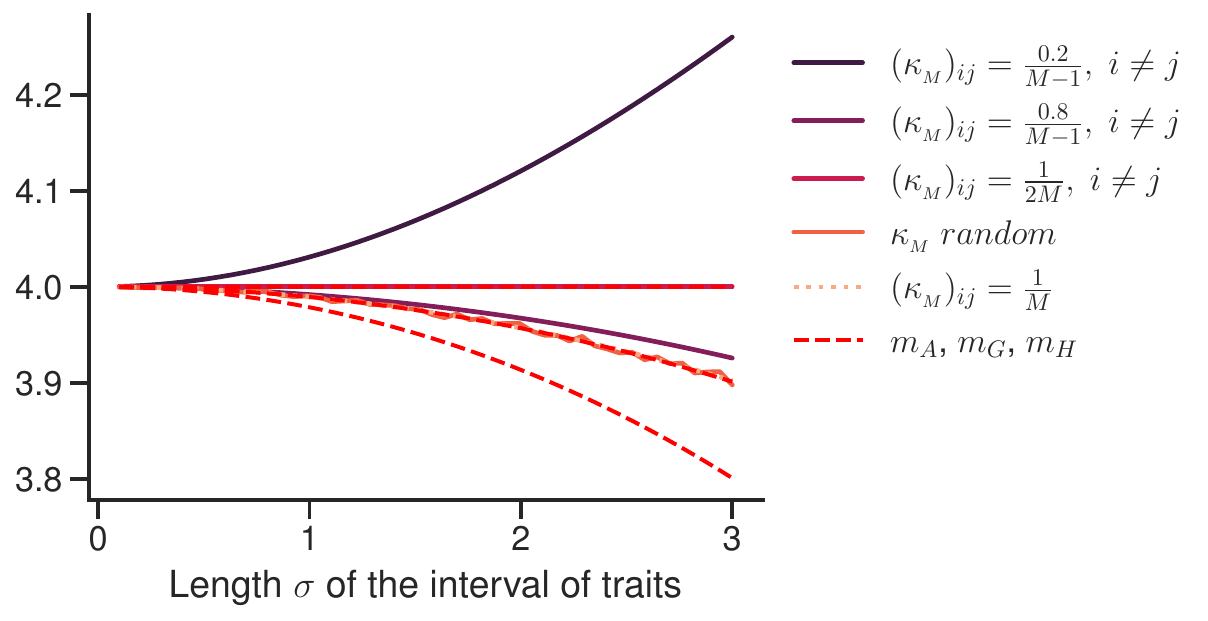}
  \caption{(Case~\ref{icaseA}).  
  Variation of the effective trait of a population of $M$ traits distributed according to~\eqref{trait_distribution} (with $m=m_A$) with respect to the range~$\sigma$ of the traits, for $M=10$ \emph{(left)} and $M=100$ \emph{(right)} and for different kernels~$\kappa_{_M}$.
  Random $\kappa_M$ \emph{(orange)} are generated by independently sampling each entry from a uniform distribution on $[0,1]$, then normalizing each row to ensure the matrix is stochastic. The case with extra-diagonal coefficients equal to \smash{$\frac{1}{2M}$} \emph{(pink)} corresponds to 
  $\kappa_{_M} = \kappa(\alpha_0)$, defined by~\eqref{eq2:def_kappa_almost_uniform} with \smash{$\alpha_0 = \frac{1}{2} + \frac{1}{2M}$}.
  It represents the \enquote{neutral} case, where increasing $\sigma$ (which measures the variability between traits) while keeping fixed the arithmetic mean of the traits, does not affect the effective trait. 
  }\label{fig:v_eff_wrt_std}
\end{figure}

\begin{multicols}{2}
\section*{Conclusion}

In this article, we have examined the impact of heterogeneous growth on the fitness of a population. To this end, we have put forth and discussed a general framework that could be adapted to other situations.
This framework involves comparing a heterogeneous population, taken in a constant environment without competition, with a fitness-identical homogeneous population. We assumed the heterogeneity to be described by a finite set of individual growth rates, each differing by a multiplicative factor (the individual trait): $\tau_i = v_i \tau$, establishing a one-to-one correspondence between growth rates $\tau_i$ and traits $v_i$. This could be seen as a discretisation of a continuous-trait heterogeneity.
In two relatively general case studies, namely constant growth and division and general division kernel (case~\ref{icaseA}), and linear growth with uniform division kernel and general division rate (case~\ref{icaseB}), explicit formulae of the effective trait (equivalently effective growth rate) as a weighted average of the traits were obtained. The weights are represented by the relative number of individuals  exhibiting a given trait. It is noteworthy that in both cases, these weights are expressed as quantities dependent on the \textit{total} population, rather than on the \textit{dividing} population or on the population at birth. We have then explored numerically these two cases.
Though the assumption of a uniform division rate is not realistic for cell division, the fact that our findings are in remarkable  accordance for the two cases analytically solvable allow us to conjecture more universality of our results, and pave the way for further research.

To reach a conclusion regarding the impact of heterogeneity, and whether it is beneficial or detrimental, further exploration of the associated \emph{costs} is necessary. This entails investigating when and how the cost of each trait is supported  by individuals. For example, is it supported by cells uniformly all along their life, i.e.\ proportional to $\int N_i$, or is it proportional to the individuals' sizes, i.e.\ proportional to $\int x N_i$, or yet  is it supported only by dividing cells, then involving quantities like $\int \tau \beta N_i$, or is it something else, such as another moment or any other weight? Combining our approach with an analysis in terms of bet-hedging strategies~\cite{de2023effective} is another perspective for future work.

\printbibliography[heading=bibintoc]

\end{multicols}

\enlargethispage{20pt}

\section*{Acknowledgments}

The authors are grateful to Sylvain Billiard and Arthur Genthon for their  reading of this work and their enlightening feedback.

\dataccess{The code and data to reproduce the figures in the paper are available at: \href{https://github.com/anais-rat/growth-fragmentation}{https://github.com/anais-rat/growth-fragmentation}.}

\competing{The authors declare that they have no competing financial interests.}

\funding{M.D. and A.R. have been partially supported by the ERC Starting Grant SKIPPER AD (number 306321)}


\disclaimer{We have not used AI-assisted technologies in creating this article.}

\newpage

\section*{Appendices}

\appendix

\section{Assumptions for a Malthusian behaviour}\label{app:assumptions}

\begin{theorem}\label{thm:Malthus}
  Assume that Case~\ref{icaseA}~\eqref{caseA} or Case~\ref{icaseB}~\eqref{caseB} holds for $\tau,\beta$ and $b$, and that $b$ satisfies~\eqref{as:b}, $\beta$ satisfies~\eqref{as:beta} and $\kappa$ satisfies~\eqref{as2:kappa}--\eqref{as1:kappa}.
  Then there exists a unique Malthus parameter $\lambda=\lambda_{(\boldsymbol{v},\kappa)}>0$, a unique nonnegative steady profile ${(N_i)}_{1\leq i\leq M} \in W_1^1(\R_+)^M$ and a unique nonnegative adjoint state $(\phi_i){_{1\leq i\leq M}} \in L^{\infty}_{loc}(0, +\infty)^M$ solution to~\eqref{eq2:N}--\eqref{eq2:phi}.
  If in addition to that,  for all $i \in \{1, \ldots, M \}$, we have $0\leq n_i^{in}(x) \leq C N_i(x)$ for some $C>0$, then the long time behaviour of the unique solution $n=(n_1,\ldots, n_M)\in \mathscr{C}(\R_+,L^1(\Phi_1(x)dx) ) \times \ldots \times \mathscr{C}(\R_+,L^1(\Phi_M(x)dx) )$ to~\eqref{eq2:SV_cauchy_pb} is given by
  \begin{equation*}
    n_i(t,x) e^{-\lambda t} 
    \underset{t\to\infty}{\longrightarrow} \rho N_i(x), 
    \qquad L^1 \big(\R_+, \phi_i(x) \dd x \big),
    \qquad  \text{for all } i \in \{1, \ldots, M \},
  \end{equation*}
  where the multiplicative factor $\rho$ is
  \begin{equation*}
  \rho= \sum_{j=1}^{M} \int_0^{\infty} n_j^{in}(x)\phi_j(x) \dd{x}.
  \end{equation*}
\end{theorem}

In the case $M=1$, a proof of Theorem~\ref{thm:Malthus} is  written in~\cite{perthame_transport_2007},  chapter 4.2, for cases which include the assumptions of Case~\ref{icaseA}, i.e.\ a constant growth rate $\tau(x)=1$, a constant division rate $\beta(x)=1$ and a general division kernel $b$ satisfying~\eqref{as:b}. A generalisation to $M\geq 2$ can be easily made from this result. 

In~\cite{rat2023growth}, a proof of Theorem~\ref{thm:Malthus} is written for cases which almost include the assumptions of Case~\ref{icaseA} and~\ref{icaseB}, except that it assumed  a division in equal sizes $b(x,y)=\delta_{x=\f{y}{2}}$. The mitosis case was studied preferentially due to the fact that interestingly, if $\tau(x)=x$, having $M\geq 2$ allowed to avoid the oscillatory behaviour observed in the case $M=1$~\cite{bernard2016cyclic}. However, the proof from~\cite{rat2023growth}  can  be generalised to include Cases~\ref{icaseA} and~\ref{icaseB}.

\section{Existence and uniqueness of the effective trait}

\begin{proposition}[Existence and uniqueness of the effective trait]\label{prop:exist}
  Let $M \geq 2$ an integer, $0<v_1< \ldots <v_M$ given traits. Assume $\tau$, $\beta$ and $b$ are such that the existence and uniqueness of the  triplet $(\lambda,N_i,\phi_i)$ with $\lambda=\lambda_{\boldsymbol{v},\kappa}>0$, ${(N_i)}_{1\leq i\leq M} \in W^{1,1}(\R_+)^M$ and ${(\phi_i)}_{1\leq i\leq M} \in L^{\infty}_{loc}(0, +\infty)^M$ solution to~\eqref{eq2:N}--\eqref{eq2:phi} is guaranteed. Then there exists a unique effective trait $v\in [v_1,v_M]$ as defined by Definition~\ref{def2:ef_fitness}.
  In Cases~\ref{icaseA} {and~\ref{icaseB}}, it satisfies 
  \begin{equation}\label{v:weightedsum:lingrowth:app}
    v=\sum\limits_{i=1}^M v_i \int_0^\infty N_i(x)\dd{x}.
  \end{equation}
\end{proposition}

\begin{proof}
Consider the eigenproblem~\eqref{eq:N:hom}~\eqref{eq:phi:hom} for $M=1$, and with parameters $\tau, \beta, b$.
For a given $v>0$, existence and uniqueness of the triplet $(\lambda_v,N_v,\phi_v)$  follows as a particular case of the eigenproblem studied in several articles~\cite{perthame_exponential_2005,M1,perthame_transport_2007,doumic_jauffret_eigenelements_2010}. We notice that for homogeneity reasons we have $\lambda_v=v\lambda_1$, hence the existence and uniqueness of $v$ such that $\lambda_v=\lambda_{\boldsymbol{v},\kappa}$: we simply define $v\coloneqq \frac{\lambda_{\boldsymbol{v},\kappa}}{\lambda_1}$. Though the proof of the inclusion $v\in [v_1,v_M]$ is already briefly sketched in~\cite[Theorem~2.3]{rat2023growth}, we detail it here for sake of completeness.
More precisely, we prove that 
\begin{equation*}
  v \geq v_1,
\end{equation*}
i.e.\ that
\begin{equation*}
  \lambda_{\boldsymbol{v},\kappa} \geq \lambda_{v_1},
\end{equation*}  
where $(\lambda_{v_1}, N_{v_1}, \phi_{v_1})$ is the solution to~\eqref{eq:N:hom}~\eqref{eq:phi:hom} with $v=v_1$.
We multiply~\eqref{eq2:N} by $\phi_{v_1}(x)$, integrate the result with respect to $x$ and sum for $i=1\ldots M$ to get
\begin{equation*}
  \begin{aligned}
    \lambda_{\boldsymbol{v},\kappa} 
    \sum_{i=1}^M\displaystyle\int_0^{\infty} N_i(x) \phi_{v_1}(x) \dd{x}
    = - v_i \sum_{i=1}^M\displaystyle\int_0^{\infty}
     (\tau N_i )'(x) \phi_{v_1}(x) \dd{x}
    - v_i \sum_{i=1}^M\displaystyle\int_0^{\infty}
     \tau(x)\beta(x)  N_i(x)  \phi_{v_1}(x) \dd{x}\\
    + 2\sum_{i=1}^M\displaystyle\int_0^{\infty}\phi_{v_1}(x) 
     \int_{x}^{\infty} \tau(y) \beta(y) b(y,x) 
     \sum_{j=1}^{M} \kappa_{ji}v_j N_j(y)\dd{y} \dd{x}
  \end{aligned}
\end{equation*}
    
Then we integrate by part the first term of the right hand side, use Fubini for the third term and we obtain
\begin{equation*}
  \lambda_{\boldsymbol{v},\kappa} 
    \sum_{i=1}^M\displaystyle\int_0^{\infty} N_i(x)\phi_{v_1}(x)\dd{x}
  = \lambda_{v_1} \sum_{i=1}^M \f{v_i}{v_1}\displaystyle\int_0^{\infty}
    N_i(x)\phi_{v_1}(x)\dd{x}
\end{equation*}
which implies the result since $v_i\geq v_1$ for all $i$.
The proof for $v\leq v_M$ is similar.

In particular, in Case~\ref{icaseA}, $\lambda_w = \beta w$ and in Case~\ref{icaseB}, $\lambda_w=w$.
The effective trait is then defined in Case~\ref{icaseA} by 
\begin{equation*}
  v= \dfrac{\lambda_{\boldsymbol{v},\kappa}}{\beta}
\end{equation*} 
and in Case~\ref{icaseB} by 
\begin{equation*}
  v=\lambda_{\boldsymbol{v},\kappa} .
\end{equation*}
Finally, {in Case~\ref{icaseA},} we {sum and} integrate~{\eqref{eq2:N:cst}} with respect to {$i$ and} $x$, which implies
\begin{equation*}
  \lambda_{\boldsymbol{v},\kappa} 
  = \beta \sum\limits_{i=1}^M v_i \int_0^\infty N_i (x)\dd{x}.
\end{equation*}
In Case~\ref{icaseB}, we plug $N_i(x)= N_i(0)e^{-\int_0^x\beta(s)ds}$ into~\eqref{eq:N:lin}, then we sum on all $i$ and integrate to 
obtain~\eqref{v:weightedsum:lingrowth:app}.
\end{proof}

\section{Proof of the formula for the weighted average for \texorpdfstring{$M=2$}{M=2}}

\begin{proposition}[General kernel, $M=2$]\label{thm2:M=2}
  Assume $M=2$ and that~\eqref{caseA} holds for $\tau,\beta$ and $b$ (Case~\ref{icaseA}), and that $b$ satisfies~\eqref{as:b}. 
  Let $k_1,\;k_2\in (0,1)$ and $\kappa$ a kernel defined by 
  \begin{equation*}
    \kappa \coloneqq
    \begin{pmatrix}
      1-{k}_1 & k_1 \\
      k_2 & 1-{k}_2 \\
    \end{pmatrix}.
  \end{equation*}
  Then $\kappa$ satisfies~\eqref{as1:kappa}\eqref{as2:kappa}, 
  and the adjoint vector $(\phi_1,\phi_2)$ solution to~\eqref{eq2:phi} is constant with respect to $x$ and the associated Malthus parameter $\lambda>0$ is
  \begin{equation}\label{def:lambda:M=2-app}
    \lambda = 
    \beta \left( \mfrac{1}{2}-k_1 \right) v_1 
    + \left( \mfrac{1}{2}-k_2 \right) v_2
    + \sqrt{{\left( \big(\mfrac{1}{2}-k_1 \big) v_1 
    - \big(\mfrac{1}{2}-k_2 \big) v_2\right)}^2 
    + 4 k_1 k_2 v_1 v_2}
    \eqqcolon \beta v;
  \end{equation}
  where $v$ is the effective trait associated with~\eqref{eq2:SV_cauchy_pb}. 
\end{proposition}

\begin{proof}
  We look for a constant solution $\phi \coloneqq (\phi_1, \phi_2) \in \R^2$, with $\phi_1,\,\phi_2>0$, to  the adjoint problem~\eqref{eq2:phi}, i.e.\
  \begin{equation}\label{eq2:SV_case1} 
    \left\{
      \begin{aligned}
        &  \big( \lambda +  {\beta} v_1\big) \phi_1 = 2 \beta
          \Big( v_1 (1-{k}_1) \phi_1 + v_1
          k_1 \phi_2 \Big),\\
        &  \big( \lambda +  { \beta} v_2\big) \phi_2 = 2 \beta
          \Big( v_2 k_2 \phi_1 + v_2 (1-{k}_2)
          \phi_2  \Big).
      \end{aligned}
    \right.
  \end{equation}

  With the notation $\alpha \coloneqq \frac{\phi_2}{\phi_1}$, this resumes to find $(\lambda, \alpha) \in {(0, \infty)}^2$ satisfying
  \begin{equation}\label{eq2:SV_case1_ter} 
    \left\{
      \begin{aligned}
        &2{ \beta} v_1 \left( \big({ \mfrac{1}{2}}-k_1 \big) + k_1 \alpha \right)
        = \lambda,\\
        &2{ \beta} v_2 \left(k_2 + \big(\mfrac{1}{2}-k_2 \big) \alpha \right)
        = \lambda \alpha.
      \end{aligned}
    \right.
  \end{equation}

  We deduce the following equation of degree $2$ on $\alpha$
  \begin{equation*}
    k_1 v_1 \alpha^2 + 
    \left( \big(\mfrac{1}{2}-k_1 \big)v_1 
    - \big(\mfrac{1}{2}- k_2 \big) v_2\right) \alpha
    - k_2 v_2 = 0
  \end{equation*}
  with real solutions (since $k_1, k_2 > 0$)
  \begin{equation*}
    \alpha_{_\pm} =
    \frac{(\f{1}{2}-k_2) v_2 - (\f{1}{2}-k_1) v_1 
    \pm \sqrt{ {\left( (\f{1}{2}-k_1) v_1 -(\f{1}{2}-k_2) v_2 \right)}^2 +
    4 k_1 k_2 v_1 v_2}}{2 k_1 v_1}.
  \end{equation*}
  Since $\alpha_{_-} < 0$ and $\alpha_{_+} > 0$, this implies $\alpha=\alpha_+$, and thus we obtain~\eqref{def:lambda:M=2}.
  It remains to prove that this defines a positive $\lambda$. If $s \coloneqq  (\f{1}{2}-k_1) v_1 + (\f{1}{2}-k_2) v_2 \geq 0$, then $\lambda >
  0$. If $s<0$, $\lambda > 0$ if and only if
  \begin{align*}
    &{\left( (\f{1}{2}-k_1)v_1 + (\f{1}{2}-k_2) v_2 \right)}^2 <
      {\left((\f{1}{2}-k_1) v_1 - (\f{1}{2}-k_2) v_2\right)}^2 + 4 k_1 k_2 v_1 v_2\\
    & \Longleftrightarrow \quad
      (\f{1}{2}-k_1)(\f{1}{2}-k_2) v_1 v_2  < k_1 k_2 v_1 v_2\\
    & \Longleftrightarrow  \quad
      k_1 + k_2 > \frac{1}{2}      
  \end{align*}
  which is verified for $s < 0$ (indeed $k_1 + k_2 \leq \frac{1}{2}$ implies that $k_1$ and $k_2$ are lower than $\frac{1}{2}$ and thus that $s \geq 0$). We have proven that $\lambda$ defined by~\eqref{def:lambda:M=2} is always positive, so by uniqueness of a solution~$(\lambda, \phi)$ to~\eqref{eq2:phi}, it is the Malthus parameter associated with~\eqref{eq2:N}-\eqref{eq2:phi}.
\end{proof}

\section{The effective trait as the (unique) positive root of a polynomial in the multimodal case with no heredity}

For a given $M$ and a set of traits $\{v_1, \ldots, v_M\}$, we introduce the
following notations 
\begin{equation}\label{def:S}
  S_0 \coloneqq 1,\qquad \,
  S_k \coloneqq \hskip-3mm 
  \sum_{ \substack{I \subset\{1,\ldots,M\} \\ \# I = k}}
  \prod_{\ell\in I}  v_{\ell} , \quad \,
  k \in \{1, \ldots, M \}.
\end{equation}

Note that we call \emph{heredity} the fact that the trait of the dividing individual influences the trait of its offspring, either positively, favoring its own trait, or negatively, favoring other traits. Hence the non-hereditary case consists in an heredity kernel which does not depend on the parental trait, namely satisfying the following assumption:
\begin{equation}\label{eq2:def_kappa_no_heredity}
  \kappa_{ij}=\kappa_j, \qquad \forall (i,j) \in {\{1, \ldots, M \}}^2.
\end{equation}
For such kernels, we obtain the following explicit expression for the effective trait.

\begin{theorem}[Effective trait as the unique positive root of an   explicit polynomial]\label{thm2:constant}
  Assume $M>0$ and that Case~\ref{icaseA}~\eqref{caseA} holds for $\tau,\beta$ and $b$, and that $b$ satisfies~\eqref{as:b}. Let $\kappa=(\kappa_{ij})$ satisfy~\eqref{as2:kappa}-\eqref{as1:kappa}, let us moreover assume that the kernel $\kappa$ satisfies~\eqref{eq2:def_kappa_no_heredity}, i.e.\ that there is no heredity at division.
  Let $(\lambda,\phi_i,N_i)$ be the unique solution  to~\eqref{eq2:N}-\eqref{eq2:phi}. Then we have $\lambda=\beta v$
  where the effective growth rate~$v$ is the unique positive root of the following polynomial, whose coefficients only depend on ${(v_i)}_{1\leq i \leq M}$ and $\kappa$
  \begin{equation}\label{eq2:lambda_uniform_m}
    P (u) = \sum_{n=0}^{M}
    \Big( S_{\texp{M-n}} - 2 \sum_{j=1}^{M} \kappa_{j}
    \msum{\substack{I \subset \{1,\ldots,M\} \setminus \{ j \} \\
        \# I=M-n}}
    \prod_{k \in I} v_k \Big) u^n
  \end{equation}
  with the conventions given by
  \begin{equation}\label{eq2:convention}
    \mprod{x \in \varnothing} x = 1,
    \qquad \sum_{x \in \varnothing} x = 0.
  \end{equation}
  In addition, the adjoint vector $\phi=(\phi_1,\ldots, \phi_M)$ solution to~\eqref{eq2:phi} is constant with respect to $x$ and is defined up to renormalization by
  \begin{equation}\label{def:phi:no_heredity}
    \phi_i : x \mapsto \frac{v_i}{v_i + v}, \qquad
    i \in \{1, \ldots, M \}.
  \end{equation}
\end{theorem}

\begin{proof}
  Looking for a non-zero solution to~\eqref{eq2:phi:cst} that is constant
  with respect to $x$ is equivalent to looking for~$(\lambda, \phi) \in (0,+\infty) \times {(0,+\infty)}^M$ solution to
  \begin{equation}\label{eq3:phi_cst}
    \lambda \phi_i = -\beta v_i \phi_i
    + 2\beta v_i \sum_{j=1}^{M} \kappa_{j} \phi_j,
    \qquad i \in \{1, \ldots, M \}.
  \end{equation}
  
  Using $v =\frac{\lambda}{\beta}$,~\eqref{eq3:phi_cst} is equivalent to
  \begin{align*}
    &\mfrac{v + v_i}{2 v_i} \phi_i
      = \sum_{j=1}^{M} \kappa_j \phi_j,
      \quad \forall i \in \{1, \ldots, M \},\\
    \Longleftrightarrow \;
    & \left\{
      \begin{aligned}
        &\mfrac{v + v_1}{2 v_1} \phi_1
          = \sum_{j=1}^{M} \kappa_j \phi_j,\\
        &\phi_i = \mfrac{v_i (v + v_1)}{v_1(v + v_i)}
          \phi_1, \quad \forall i \in \{2, \ldots, M \},
      \end{aligned} \right.\\
    \Longleftrightarrow \;
    & \left\{
      \begin{aligned}
        & \phi_1
          = 2 \sum_{j=1}^{M} \kappa_j \mfrac{v_j}{(v + v_j)} \phi_1,\\
        &\phi_i = \mfrac{v_i (v + v_1)}{v_1(v + v_i)}
          \phi_1, \quad \forall i \in \{2, \ldots, M \},
      \end{aligned} \right.
  \end{align*}
  We multiply the first line by~\smash{$\prod_{\texp{\, 1\leq k \leq M}} (v + v_k)$}and obtain 
  \begin{equation}\label{inter}
    \left\{
    \begin{aligned}
      &P(v) \phi_1 = 0,\\
      &\phi_i = \mfrac{v_i (v + v_1)}{v_1(v + v_i)}
      \phi_1, \quad \forall i \in \{2, \ldots, M \},
    \end{aligned}
    \right.
  \end{equation}
  where~$P$ is the polynomial defined by:
  \begin{equation*}
    \begin{aligned}
      P(u) &= 2 \sum_{j=1}^{M} \Big( \kappa_j v_j
             \mmprod{\substack{k=1 \\ k \neq j}}{M} (u + v_k) \Big)
      - \mmprod{k=1}{M} (u + v_k),\\
           &= 2 \sum_{j=1}^{M} \kappa_j \mmprod{k=1}{M} (u + v_k)
             - 2 u \sum_{j=1}^{M} \kappa_j
             \mmprod{\substack{k=1 \\ k \neq j}}{M} (u + v_k) 
      - \mmprod{k=1}{M} (u + v_k),\\
           &= \mmprod{k=1}{M} (u + v_k)
             - 2u \sum_{j=1}^{M} \kappa_j
             \mmprod{\substack{k=1 \\ k \neq j}}{M} (u + v_k),\\
           &= \sum_{n=0}^{M} \Big( S_{\texp{M-n}} - 2 \sum_{j=1}^{M}  \kappa_{j}
             \msum{\substack{I \subset \{1,\ldots,M\} \setminus \{ j \} \\
                             \# I=M-n}}
      \prod_{k \in I} v_k \Big) u^n.
    \end{aligned}
  \end{equation*}
  From the first equality of~\eqref{inter}, either~$\phi_1=0$ or $P(v)=0$. If $\phi_1=0$, the second line of~\eqref{inter} implies that for all $j \in \{1,\ldots,M \}$ we also have $\phi_j=0$, i.e.\ $\phi \equiv 0$.  By contradiction, $P(v)=0$.  
  Then, if $(\beta v, \Phi)$  is a constant solution to~\eqref{eq2:phi:cst}, then $P(v)=0$ and $\Phi_i= \frac{v_i(v+v_1)}{v_1(v+v_i)} \Phi_1$. Conversely, for
  any $v>0$ such that $P(v)=0$, then there exists $\lambda=\beta v$ and
  $\phi=(\phi_1,\ldots, \phi_M) \in {(0,+\infty)}^M$ such that~\eqref{eq3:phi_cst} is satisfied. 
  By uniqueness of the adjoint problem~\eqref{eq2:phi:cst}, this proves that $P$ has at most one positive root.  Recalling the conventions~\eqref{eq2:convention}, the term in front of $u^M$ is
  \begin{equation*}
    S_0 - 2 \sum_{j=1}^M \kappa_j  \prod_{k \in \varnothing} v_k
    = 1 - 2 \sum_{j=1}^M \kappa_j = -1,
  \end{equation*}
  and the constant term is
  \begin{equation*}
    S_M - 2 \sum_{j=1}^M \kappa_j \sum_{I \in \varnothing}
    \prod_{k \in I} v_k
    = S_M.
  \end{equation*}
  Therefore, we have
  \begin{equation*}
    P(+\infty) = -\infty, \qquad P(0)>0,
  \end{equation*}
  so that $P$ has at least one positive root $v$.  

\end{proof}

\section{Mother-daughter Pearson correlation coefficient}\label{sec:cor}

Denote by $V_m$ and $V_d$ the discrete random variables representing the trait of a mother cell and that of one of its daughter, respectively.
The law of $V_d$ conditionally on $V_m$ is prescribed by the choice of the heredity kernel $\kappa$ in our model, via:
\begin{equation}\label{eq:link_P_kappa}
  \kappa_{ij} \coloneqq \bP( V_d=v_j \mid V_m=v_i),
  \qquad 1 \leq i, j, \leq M.
\end{equation}
The question we address is: how to build a family of kernels $\kappa$ indexed by a 1D parameter~$\alpha$ that reflects the strength of inheritance --- more specifically, the correlation between the mother and daughter traits, as quantified by the Pearson correlation coefficient $\Gamma \coloneqq \Cor(V_d,V_m)$ between $V_d$ and $V_m$, given by:
\begin{equation}\label{eq:def_corr}
  \Gamma \coloneqq \frac{\Cov(V_d, V_m)}{\sqrt{\Var(V_d) \Var(V_m)}}.
\end{equation}
In fact, the heredity kernels defined in~\eqref{eq2:def_kappa_almost_uniform}, namely:
\begin{equation}\label{eqA:def_kappa_almost_uniform}
  \kappa_{ii} (\alpha) =  \alpha, \qquad
  \kappa_{ij} (\alpha) = \frac{1- \alpha}{M-1}, 
  \quad i \neq j,
\end{equation}
provide a natural and effective choice, as stated in the following result.

\begin{proposition}\label{prop:cor}
  Let $V_m$ and $V_d$  be the random variables over $\{v_1, \ldots, v_M \}$ representing the trait of a mother cell and of one of its daughters. Let $\alpha \in (0, 1)$ be the parameter defining the heredity kernel via~\eqref{eqA:def_kappa_almost_uniform}, and therefore the law of $V_d$ conditionally on $V_m$. 

  Then the Pearson correlation coefficient $\Gamma=\Gamma(\alpha)$ between $V_d$ and $V_m$, defined by~\eqref{eq:def_corr}, is a strictly increasing function of $\alpha$. Moreover, $\alpha = \frac{1}{M}$ corresponds to the case where there is no correlation between mother and daughter traits, that is~$\Gamma \big(\frac{1}{M} \big) = 0$, and $\alpha= 1$ to a perfect positive correlation, i.e.\ $\Gamma(1)=1$.

  When $V_m$ follows the uniform distribution on $\{v_1, \ldots, v_M \}$, the same holds for $V_d$, and the correlation is linear in $\alpha$, given by $\Gamma(\alpha) = \frac{\alpha M -1}{M-1}$.
\end{proposition}
This result confirm that $\alpha$ provides a meaningful and interpretable proxy for the level of trait correlation across generations.

\begin{proof}
  We start noticing that a kernel $\kappa(\alpha)$ defined by~\eqref{eqA:def_kappa_almost_uniform} with $\alpha \in (0, 1)$ can actually be described by the following modelling choice for the law of $V_d$ conditional on $V_m$:
  \begin{equation}\label{eq:model_proba_V}
    V_d = B V_m + (1-B)  W_m,
  \end{equation}
  where $B$ and $W_m=W(V_m)$ are independent random variables, 
  $B \sim \textrm{Bern}(\alpha)$ is a Bernoulli random variable with success probability $\alpha$, and $W_m=W(V_m)$ is a random variable whose law, conditionally on $V_m=v_i$, is uniform over the set $\{v_1, \ldots, v_m \} \setminus v_i$.
  Indeed, one can easily verify that $\kappa= \kappa(\alpha)$  defined by~\eqref{eqA:def_kappa_almost_uniform} and $\bP$ defined by~\eqref{eq:model_proba_V} satisfy relation~\eqref{eq:link_P_kappa}. For $i \neq j$, using the independence of $B$ from $V_m$ and $W_m$, and their definition, we obtain
    \begin{equation*}
     \bP( V_d = v_j \mid V_m = v_i ) 
     = \bP(B=0, W_m=v_j \mid V_m = v_i)
     = \bP(B=0) \bP(W_m=v_j)
     = \mfrac{1-\alpha}{M-1}
     = \kappa_{ij}(\alpha),
    \end{equation*}
    and similarly for $i=j$
    \begin{equation*}
     \bP( V_d = v_i \mid V_m = v_i )
     = \bP(B=1) + 0
     = \alpha
     =\kappa_{ii} (\alpha).
   \end{equation*}
  Note that both~\eqref{eqA:def_kappa_almost_uniform} and \eqref{eq:model_proba_V} can be interpreted as follows: with probability $\alpha$, the daughter inherits the trait of the mother ($V_d=V_m$); otherwise, the daughter's trait is drawn uniformly from the set of all traits excluding the mother's trait.
  
  Let us now compute $\Gamma(V_d, V_m)$. We start by computing some useful quantities. Denote by $R \coloneqq \frac{1}{M-1}$, we have
  \begin{equation*}
    \bP(W_m=v_i, V_m=v_j)
    = \bP(W_m=v_i \mid V_m = v_j) \bP(V_m=v_j)
    = \left\{ 
    \begin{aligned}
      &R  \,\bP(V_m=v_j), \; &&\text{if~} i \neq j,\\[-5pt]
      &0, \; &&\text{if~} i = j,
    \end{aligned} \right.
  \end{equation*}
  and for $k, \ell \geq 0$ and $U$ a discrete random variable having uniform distribution over $\{v_1, \ldots v_M \}$
  \begin{equation}\label{eq:moments}
    \begin{aligned}
      \bE[W_m^k V_m^\ell] 
      &= \sum_{1 \leq i, j \leq M} v_i^k v_j^\ell
      \, \bP(W_m=v_i, V_m=v_j)
      = \! \! \! \sum_{1\leq i \neq j \leq M}  \! \! \!
      v_i^k v_j^\ell R \, \bP(V_m=v_j)
      = \sum_{1 \leq i \leq M} v_i^k R
      \Big( \bE[V_m^\ell] - v_i^\ell \bP(V_m=v_i) \Big)\\
      &= R \Big( M \bE[U^k] \bE[V_m^\ell] - \bE[V_m^{k+\ell}] \Big).
    \end{aligned}
  \end{equation}
  By independence of $B$ and $W_m$ and of $B$ and $V_m$ we have:
  \begin{align*}
    \bE [ V_d ] &= \bE [ B V_m +  (1-B)W_m ]
    = \alpha \bE [V_m ]
    + (1-\alpha) \bE [W_m],
  \end{align*}
  and using additionally that $B=B^2$ and $B(1-B)=0$, since $B$ is $0$ or $1$, we have also
  \begin{align*}
    \bE \big[ V_d^2 \big]
    = \bE \big[B^2 \big] \bE \big[V_m^2 \big]
     + 2 \bE \big[ B (1-B) \big] \bE \big[ W_m V_m \big]
     + \bE \big[ (1- B)^2 \big] \bE \big[W_m^2 \big]
    = \alpha \bE \big[V_m^2 \big]
    + (1 - \alpha) \bE \big[W_m^2 \big].
  \end{align*}
  Thus, by denoting $\Delta = \bE [ W_m - V_m ]$
  \begin{equation}\label{eq:var_Vd}
    \begin{aligned}
      \Var(V_d)
      &= \bE[V_d^2] - \bE[V_d]^2
      = \alpha \bE \big[V_m^2 \big]
      + (1 - \alpha) \bE \big[W_m^2 \big] -\alpha^2 \bE [V_m ]^2 
      - 2 \alpha (1-\alpha) \bE [V_m ] \bE [W_m]
      - (1-\alpha)^2 \bE [W_m]^2\\
      &= \alpha \Var(V_m) + \alpha(1-\alpha) \bE[V_m]^2 
      +(1-\alpha) \Var(W_m) + \alpha (1-\alpha) \bE [W_m]^2
      - 2 \alpha (1-\alpha) \bE [V_m ] \bE [W_m]\\
      &= \alpha \Var(V_m) 
      + (1-\alpha) \Var(W_m) 
      + \alpha (1-\alpha) \Delta^2.
    \end{aligned}
  \end{equation}
  Similarly, 
  \begin{align*}
    \Cov(V_d, V_m) 
    &= \bE[V_d V_m] - \bE[V_d] \bE[V_m]
    = \alpha \bE[V_m^2] + (1-\alpha) \bE[ W_m V_m] 
    - \alpha \bE [V_m ]^2 - (1-\alpha) \bE [W_m]\bE [V_m ]\\
    &= \alpha \Var(V_m) + (1- \alpha) \Cov(W_m, V_m).
  \end{align*}
  Formulae~\eqref{eq:moments} yields
  \begin{align*}
    \Cov(W_m, V_m) 
    = \bE[W_m V_m] - \bE[W_m] \bE[V_m]
    = R \big( M \bE[U] \bE[V_m] - \bE[V_m^2] \big)
    - R \big(M \bE[U] - \bE[V_m] \big) \bE[V_m]
    = - R \Var(V_m)
  \end{align*}
  and therefore
  \begin{equation*}
    \Cov(V_d, V_m)
    = \big(\alpha(1+R) -R \big) \Var(V_m)
    = (\alpha M - 1) R \Var(V_m).
  \end{equation*}
  Finally, we obtain the Pearson correlation coefficient
  \begin{align*}
    \Gamma(\alpha)
    &= \frac{\Cov(V_d, V_m)}{\sqrt{\Var(V_d) \Var(V_m)}}
    = (\alpha M -1) R 
     \sqrt{\frac{\Var(V_m)}{\Var(V_d)(\alpha)}}.
  \end{align*}
  In particular $\Gamma(\frac{1}{M})=0$ and, using also~\eqref{eq:var_Vd}, $\Gamma(1)=1$. Let us compute the derivative, making all $\alpha$-dependencies explicit:
  \begin{align*}
    \Gamma'(\alpha)
    &= \frac{R}{2 \Var(V_d)(\alpha)} 
    \sqrt{\frac{\Var(V_m)}{\Var(V_d)(\alpha)}}
    \Big( M {\Var(V_d)(\alpha)} 
    - (\alpha M - 1) \Var(V_d)'(\alpha) \Big).
  \end{align*}
  According to~\eqref{eq:var_Vd}, the function $\alpha \mapsto \Var(V_d)(\alpha)$ is a downward-opening parabola. It attains its maximum at $\alpha = \frac{1}{M}$ since then $V_d$ is the uniform law on the set of traits $\cV$, which has maximal variance among all probability distributions on~$\cV$.
  It follows that $\Var(V_d)'$ is positive on $(0, \frac{1}{M})$ and negative on $(\frac{1}{M},1)$. Moreover, $\alpha \mapsto \alpha M -1$ has opposite sign on $(0, 1)$, hence the product $(\alpha M-1) \Var(V_d)'(\alpha)$ is non-negative for all $\alpha \in (0, 1)$, positive for $\alpha \neq \frac{1}{M}$.
  We conclude that the mother-daughter Pearson correlation $\Gamma$ is an increasing function of $\alpha$.

  When $V_m$ and $U \sim \mathcal{U}_{\{v_1, \ldots v_M \}}$ are equal in law, we have
  \begin{equation*}
    \bP(W_m = v_i) 
    = \sum_{1 \leq j \leq M} \bP(W_m=v_i \mid V_m = v_i) \bP(V_m = v_i)
    = \sum_{j \neq i} \frac{1}{M-1} \cdotp \frac{1}{M}
    = \frac{1}{M}.
  \end{equation*}
  Thus, $W_m$ also follows the uniform distribution on the set of traits. Similarly we can obtain that $V_d$ does as well. Using the previous formulae, we conclude that
  \begin{equation*}
    \Gamma(\alpha)
    = (\alpha M -1) R = \frac{\alpha M -1}{M-1}.
  \end{equation*}
\end{proof}

\newpage

\setcounter{figure}{0}
\renewcommand{\figurename}{Fig.}
\renewcommand{\thefigure}{S\arabic{figure}}

\section*{Supplementary Information}

\begin{figure}[H]
  \centering
  \begin{subfigure}[t]{0.35\textwidth}
    \centering
    \includegraphics[width=\textwidth]{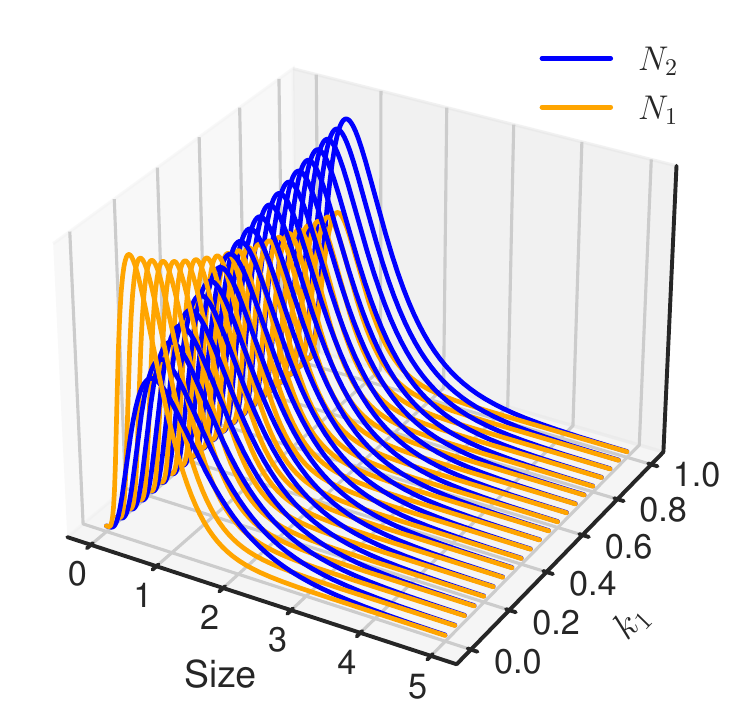}
    \caption{$k_2=0.2$}\label{sfig:a}
  \end{subfigure}%
  ~ 
  \begin{subfigure}[t]{0.35\textwidth}
    \centering
    \includegraphics[width=\textwidth]{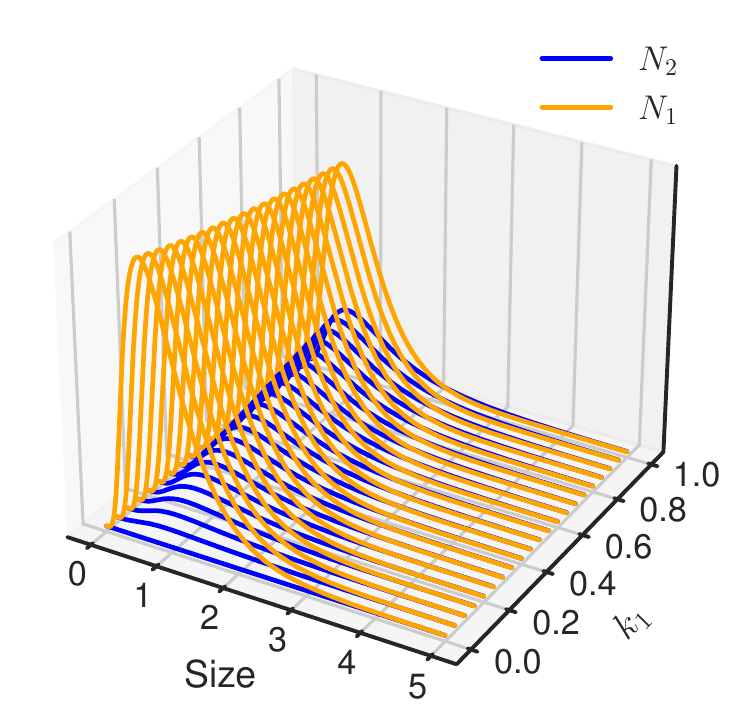}
    \caption{$k_2=0.8$}\label{sfig:b}
  \end{subfigure}
  ~
  \begin{subfigure}[t]{0.35\textwidth}
    \centering
    \includegraphics[width=\textwidth]{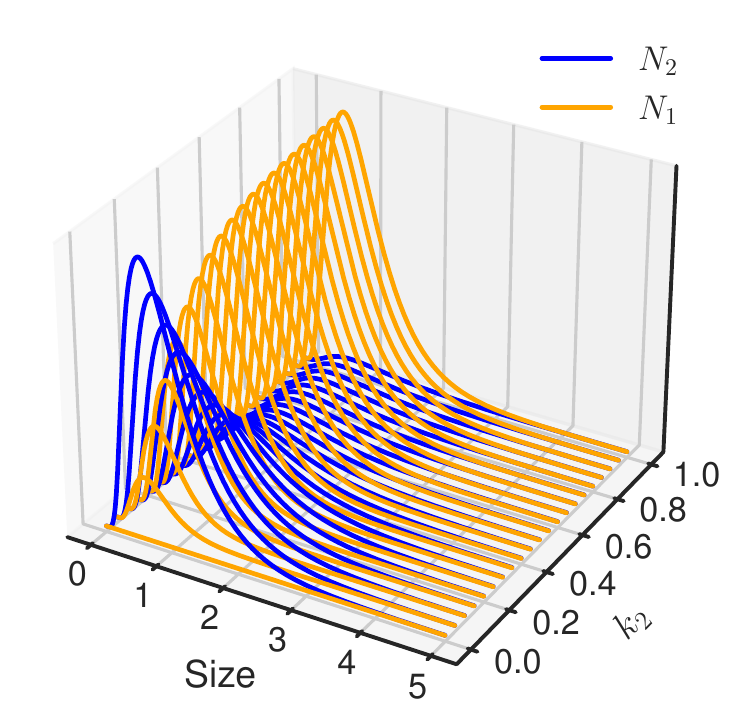}
    \caption{$k_1=0.2$}\label{sfig:c}
  \end{subfigure}
  ~
  \begin{subfigure}[t]{0.35\textwidth}
    \centering
    \includegraphics[width=\textwidth]{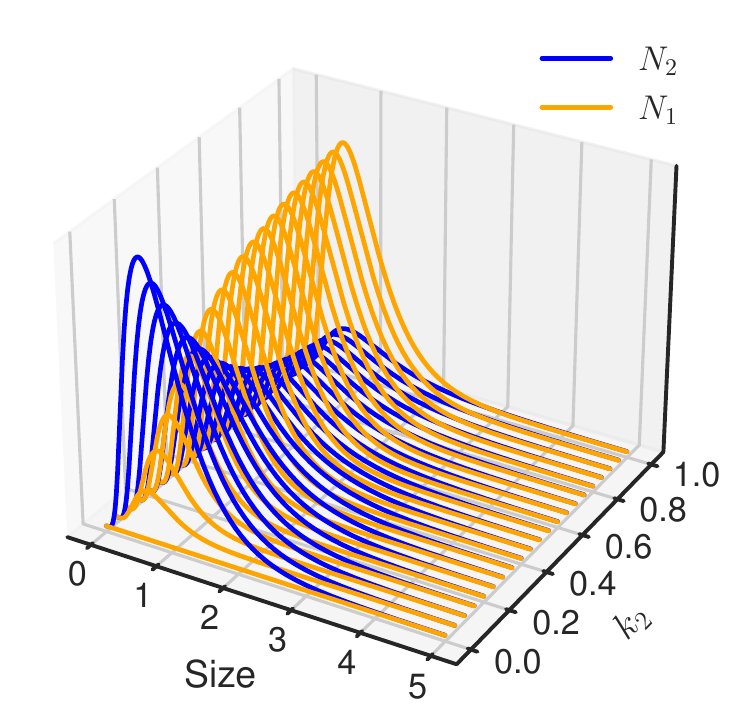}
    \caption{$k_1=0.2$\label{sfig:d}}
  \end{subfigure}
  \caption{Cases~\ref{icaseA} (and~\ref{icaseB}) $M=2$. Homeostatic size distributions {$\boldsymbol{N}=(N_1, N_2)$} in a heterogeneous population with traits $\boldsymbol{v} = (0.5, 2.5)$ and heredity kernel defined by $k_1$ and $k_2$ via~\eqref{def:kappa_m2}, with respect to either varying $k_1$ at $k_2$ fixed (\emph{top}), or varying $k_2$ at fixed $k_1$  (\emph{bottom}). Figures~\ref{sfig:c} and~\ref{sfig:d} illustrate that when the larger trait $v_2$ is strongly inherited ($k_2$ close to zero), $v_1$ is poorly represented, regardless of how it is transmitted (via $k_1$). In contrast, we see on~\ref{sfig:a} and~\ref{sfig:b} that strong inheritance of $v_1$ (i.e.\ $k_1$ close to zero) is not sufficient to outcompete $v_2$ if trait is likely enough to be self-transmitted (i.e.\ $k_2$ is small enough).}\label{figS6}
\end{figure}

\vfill~

\begin{figure}[H]
  \centering
  \begin{subfigure}[t]{0.29\textwidth}
    \centering
    \includegraphics[height=3.8cm, trim={0 0 5.7cm 0}, clip=true]{%
      v_and_means_M10_wrt_std_n_correlationsp01-p8-9_small_CVari_constant-1_no-title_w-means.pdf}

    \includegraphics[height=3.8cm, trim={0 0 5.7cm 0}, clip=true]{%
      v_and_means_M10_wrt_std_n_correlationsp01-p8-9_small_CVgeo_constant-1_no-title_w-means.pdf}
        
    \includegraphics[height=3.8cm, trim={0 0 5.7cm 0}, clip=true]{%
      v_and_means_M10_wrt_std_n_correlationsp01-p8-9_small_CVhar_constant-1_no-title_w-means.pdf}
    \caption{Cases~\ref{icaseA} and~\ref{icaseB}}
  \end{subfigure}%
  ~ 
  \begin{subfigure}[t]{0.29\textwidth}
    \centering
    \includegraphics[height=3.8cm, trim={1cm 0 5.7cm 0}, clip=true]{%
      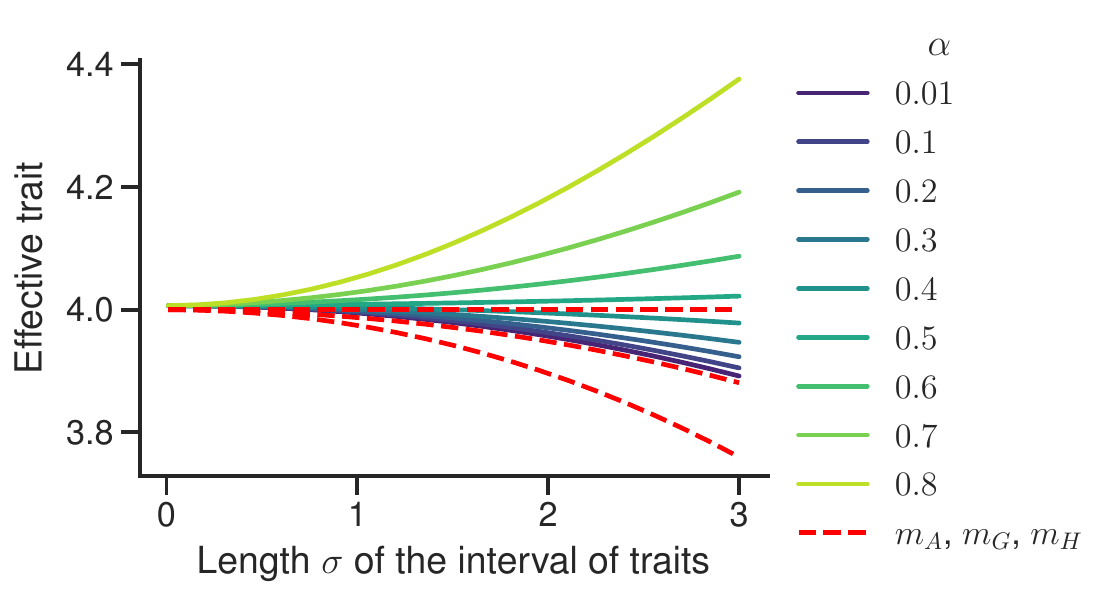}
  
    \includegraphics[height=3.8cm, trim={1cm 0 5.7cm 0}, clip=true]{%
      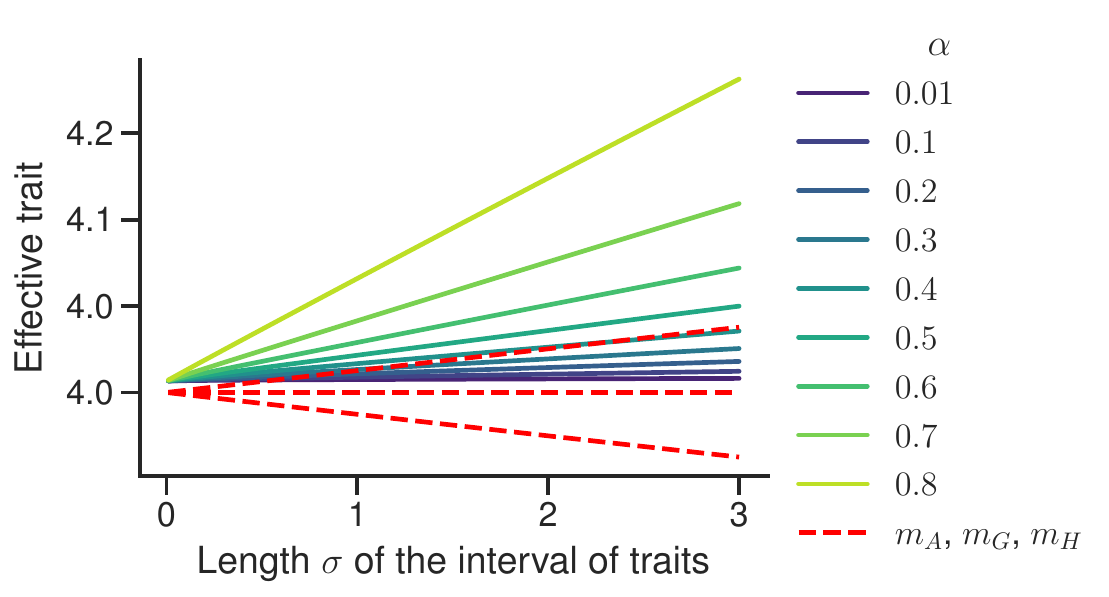}
  
    \includegraphics[height=3.8cm, trim={1cm 0 5.7cm 0}, clip=true]{%
      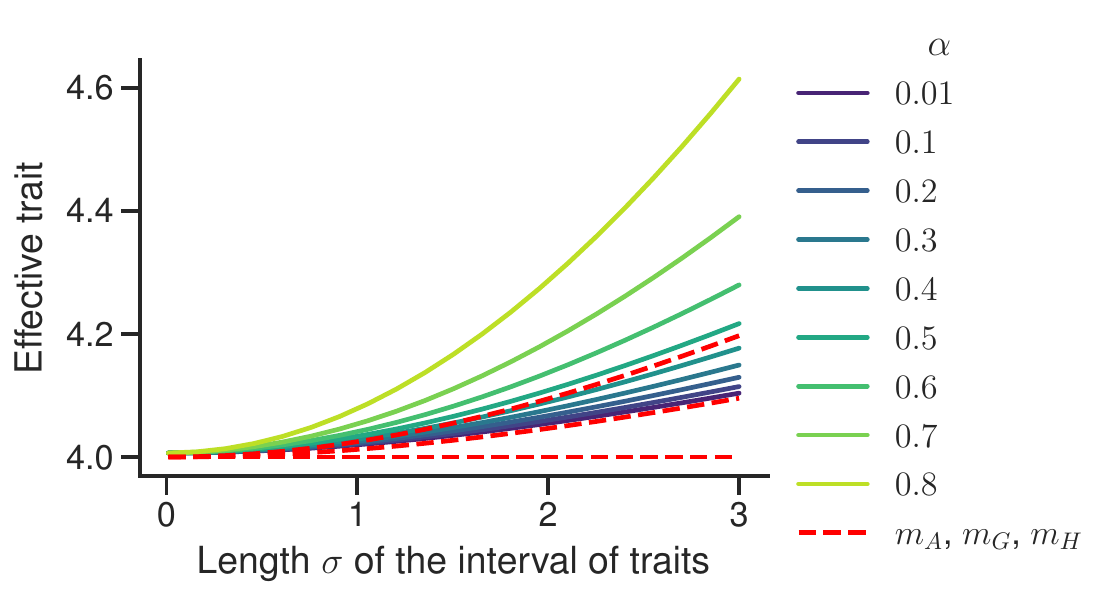}
    \caption{Equal mitosis, $\tau \equiv x, \beta \equiv 1$}
  \end{subfigure}
  ~
  \begin{subfigure}[t]{0.35\textwidth}
    \centering
    \includegraphics[height=3.8cm, trim={1cm 0 0cm 0}, clip=true]{%
      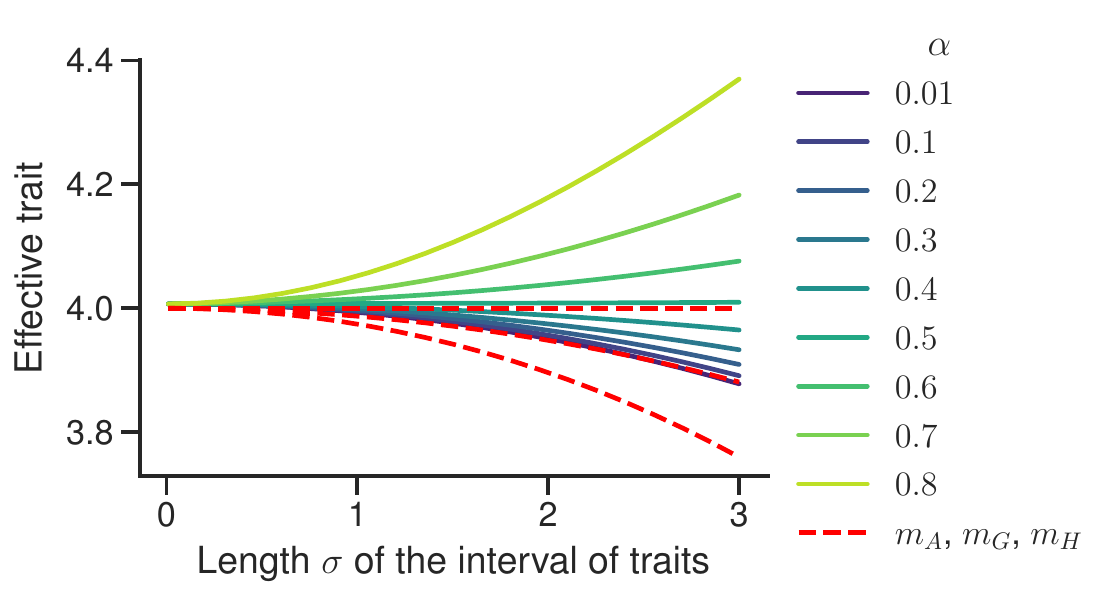}
  
    \hspace{-1cm}\includegraphics[height=3.8cm, trim={1cm 0 5.7cm 0}, clip=true]{%
      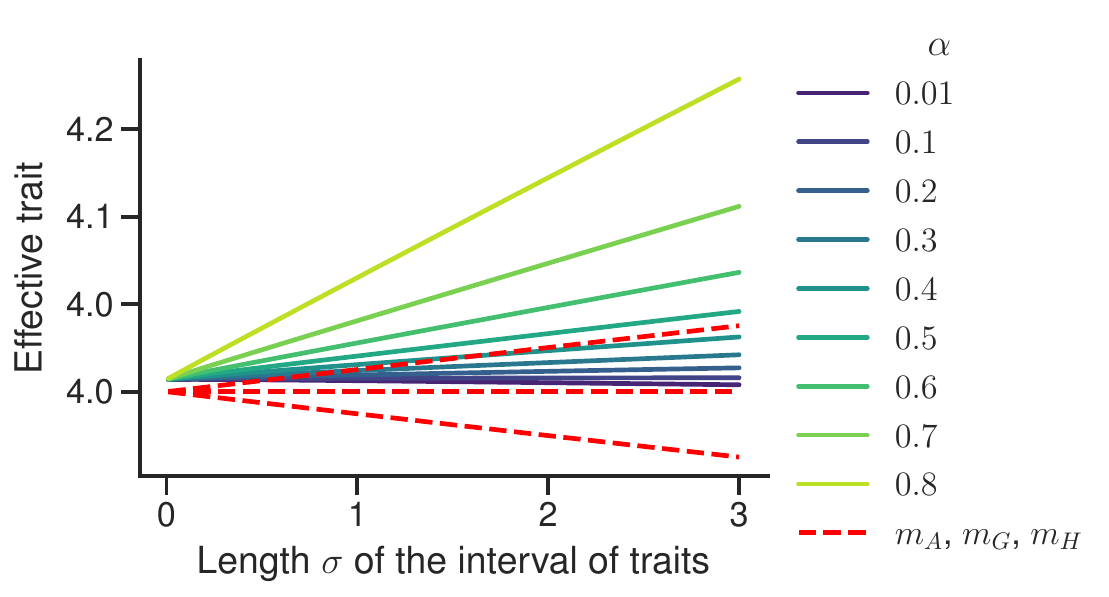}
  
    \hspace{-1cm}\includegraphics[height=3.8cm, trim={1cm 0 5.7cm 0}, clip=true]{%
      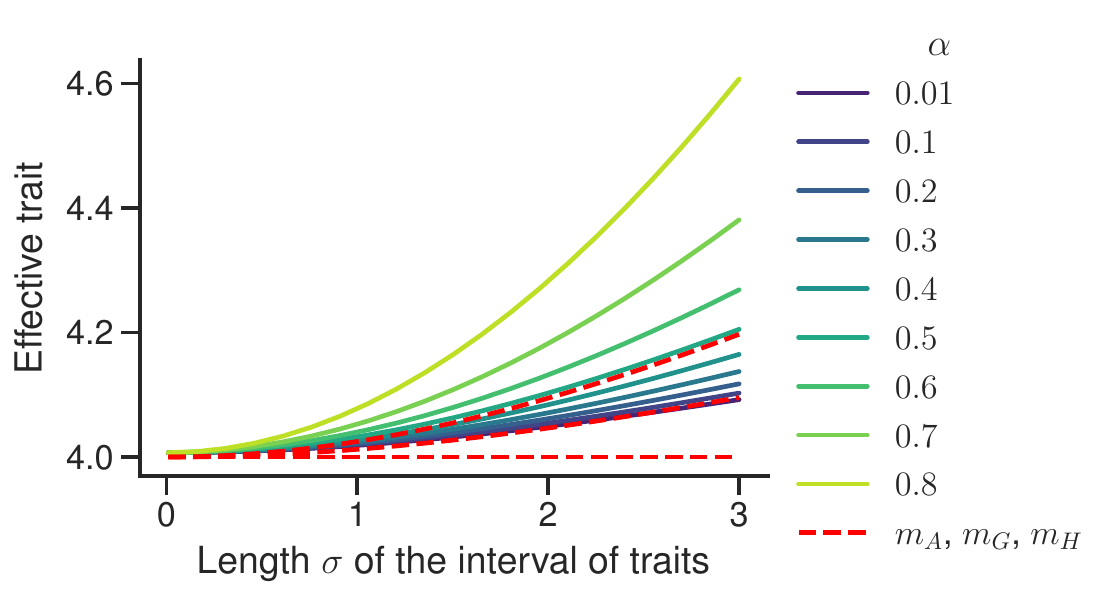}
    \caption{Equal mitosis, $\tau \equiv x, \beta \equiv x$}
  \end{subfigure}
  \caption{Variation of the effective trait of a population of $M=10$ traits distributed as described in~\eqref{trait_distribution} with $m=m_A$ for the figures on top, $m=m_G$ for the figures in the middle and $m=m_H$ for the figures on the bottom. The parameters are $\Bar{v}=4$, and kernel~$\kappa(\alpha)$ satisfying~\eqref{eq2:def_kappa_almost_uniform}, with respect to~$\sigma$ and for different values of $\alpha$.}\label{fig:v_eff_wrt_std_correlation_all}
\end{figure}

\end{document}